\documentclass[11pt]{article}
\usepackage{mathrsfs}
\usepackage{amssymb}
\usepackage[hang,flushmargin]{footmisc}
\usepackage{multirow}
\usepackage{amsfonts}
\usepackage{graphicx}
\usepackage{amsmath,amsthm,amssymb,latexsym,color}
\usepackage[colorlinks, linkcolor=blue, anchorcolor=gree, citecolor=red]{hyperref}
\usepackage[numbers,sort&compress]{natbib}
 \makeatletter
    
    \newcommand{\Rmnum}[1]
    {\expandafter\@slowromancap\romannumeral #1@}
    \makeatother
\def\wz{\tilde}

\newtheorem{thm}{Theorem}[section]

\newtheorem{prop}[thm]{Proposition}
\newtheorem{lemma}[thm]{Lemma}
\newcounter{foo}[section]

\newtheorem{step}[foo]{Step}

\newtheorem{example}[thm]{Example}
\newtheorem{defin}[thm]{Definition}

\theoremstyle{definition}
\newtheorem{remark}[thm]{Remark}

\usepackage{CJK}
\begin{document}
\begin{CJK*}{GBK}{song}

\renewcommand{\baselinestretch}{1.3}
\title{Weakly distance-regular digraphs of one type of arcs}
\author{Yushuang Fan\qquad Zhiqi Wang\qquad Yuefeng Yang\footnote{Corresponding author.}
\\
{\footnotesize \em  School of Science, China University of Geosciences, Beijing, 100083, China}  }
\date{}
\maketitle
\footnote{\scriptsize
{\em E-mail address:} fys@cugb.edu.cn (Yushuang Fan), Shero577@163.com (Zhiqi Wang), yangyf@cugb.edu.cn (Yuefeng Yang).}

\begin{abstract}

In this paper, we classify all commutative weakly distance-regular digraphs of girth $g$ and one type of arcs under the assumption that $p_{(1,g-1),(1,g-1)}^{(2,g-2)}\geq k_{1,g-1}-2$. In consequence, we recover \cite[Theorem 1.1]{YYF18} as a special case of our result.



\end{abstract}

\section{Introduction}

All the digraphs considered in this paper are finite, simple, strongly connected and not undirected. Let $\Gamma$ be a digraph and $V\Gamma$ be its vertex set. For any $x,y\in V\Gamma$, let $\partial(x,y)$ denote the \emph{distance} from $x$ to $y$ in $\Gamma$. The pair $\wz{\partial}(x,y)=(\partial(x,y),\partial(y,x))$ is called the \emph{two-way distance} from $x$ to $y$. Let $\wz{\partial}(\Gamma)$ be the set of all pairs $\wz{\partial}(x,y)$. An arc $(u,v)$ of $\Gamma$ is of \emph{type} $(1,q-1)$ if $\partial(v,u)=q-1$.

As a natural generalization of distance-regular graphs (see \cite{AEB98, DKT16} for the theory of distance-regular graphs), Wang and Suzuki \cite{KSW03} introduced the concept of weakly distance-regular digraphs. A digraph $\Gamma$ is said to be \emph{weakly distance-regular} if, for any $\wz{h}$, $\wz{i}$, $\wz{j}\in\wz{\partial}(\Gamma)$, the number of $z\in V\Gamma$ such that $\wz{\partial}(x,z)=\wz{i}$ and $\wz{\partial}(z,y)=\wz{j}$ is constant whenever $\wz{\partial}(x,y)=\wz{h}$. This constant is denoted by $p_{\wz{i},\wz{j}}^{\wz{h}}$. The integers $p_{\wz{i},\wz{j}}^{\wz{h}}$ are called the \emph{intersection numbers}. The size of $\Gamma_{\tilde{i}}(x):=\{y\in V\Gamma\mid\tilde{\partial}(x,y)=\tilde{i}\}$
depends only on $\tilde{i}$, denoted by $k_{\tilde{i}}$. We say
that $\Gamma$ is \emph{commutative} if $p_{\tilde{i},\tilde{j}}^{\tilde{h}}=p_{\tilde{j},\tilde{i}}^{\tilde{h}}$ for all $\tilde{i}$, $\tilde{j}$,
$\tilde{h}\in\tilde{\partial}(\Gamma)$.

Weakly distance-transitive digraphs is a special class of weakly distance-regular digraphs. A digraph $\Gamma$ is said to be {\em weakly distance-transitive} if, for any vertices $x,y,x'$ and $y'$ satisfying $\wz{\partial}(x,y)=\wz{\partial}(x',y')$, there exists an automorphism $\sigma$ of $\Gamma$ such that $x'=\sigma(x)$ and $y'=\sigma(y)$.

Beginning with \cite{KSW03}, some special families of weakly distance-regular digraphs were classified. See \cite{KSW03,HS04} for valency $2$, \cite{KSW04,YYF16,YYF18} for valency $3$, \cite{HS04} for thin case, \cite{YYF20} for quasi-thin case, and \cite{YYF} for thick case. Especially, in \cite{YYF18}, the third author, Lv and Wang determined all commutative weakly distance-regular digraphs of valency $3$ and one type of arcs. In this paper, we continue to study weakly distance-regular digraphs of one type of arcs, and obtain the following result which improves \cite[Theorem 1.1]{YYF18}.

\begin{thm}\label{Main3}
Let $\Gamma$ be a commutative weakly distance-regular digraph of valency $k$ more than $3$ and girth $g$. Suppose that $\Gamma$ has one type of arcs, that is, $k=k_{1,g-1}$. If $p_{(1,g-1),(1,g-1)}^{(2,g-2)}\geq k-2$, then $\Gamma$ is isomorphic to one of the following digraphs:
\begin{itemize}
\item[{\rm (i)}] ${\rm Cay}(\mathbb{Z}_{11},\{1,3,4,5,9\})$;

\item[{\rm (ii)}] ${\rm Cay}(\mathbb{Z}_{14},\{1,2,4,8,9,11\})$;

\item[{\rm (iii)}] ${\rm Cay}(\mathbb{Z}_{26},\{1,3,9,14,16,22\})$;

\item[{\rm (iv)}] ${\rm Cay}(\mathbb{Z}_{3}\times\mathbb{Z}_{k+1},\{(1,1),(1,2),\ldots,(1,k)\})$;

\item[{\rm (v)}] ${\rm Cay}(\mathbb{Z}_{g}\times\mathbb{Z}_k,\{(1,0),(1,1),\ldots,(1,k-1)\})$.
\end{itemize}
\end{thm}

In order to give a high-level description of our proof of Theorem \ref{Main3}, we need
additional notations and terminologies. Let $\Gamma$ be a weakly distance-regular digraph and $R=\{\Gamma_{\wz{i}}\mid\wz{i}\in\wz{\partial}(\Gamma)\}$, where  $\Gamma_{\wz{i}}=\{(x,y)\in V\Gamma\times V\Gamma\mid\wz{\partial}(x,y)=\wz{i}\}$. Then $(V\Gamma,R)$ is an association scheme (see \cite{EB84,PHZ96,PHZ05} for the theory of association schemes), which is called
the \emph{attached scheme} of $\Gamma$. For each $(i,j)\in\wz{\partial}(\Gamma)$, let $A_{i,j}$ denote a matrix with rows and columns indexed by $V\Gamma$ such that $(A_{i,j})_{x,y}=1$ if $\wz{\partial}(x,y)=(i,j)$, and $(A_{i,j})_{x,y}=0$ otherwise.  It follows from the definition of association schemes that
\begin{eqnarray}
A_{i,i'}A_{j,j'}=\sum_{(h,h')\in\wz{\partial}(\Gamma)}p_{(i,i'),(j,j')}^{(h,h')}A_{h,h'}.\nonumber
\end{eqnarray}
For two nonempty subsets $E$ and $F$ of $R$, define
\begin{eqnarray}
EF=\{\Gamma_{\wz{h}}\mid\sum_{\Gamma_{\wz{i}}\in E}\sum_{\Gamma_{\wz{j}}\in F}p_{\wz{i},\wz{j}}^{\wz{h}}\neq0\}, \nonumber
\end{eqnarray}
and write $\Gamma_{\wz{i}}\Gamma_{\wz{j}}$ instead of $\{\Gamma_{\wz{i}}\}\{\Gamma_{\wz{j}}\}$.  For any $(a,b)\in\wz{\partial}(\Gamma)$, we usually write $\Gamma_{a,b}$ (resp. $k_{a,b}$) instead of $\Gamma_{(a,b)}$ (resp. $k_{(a,b)}$).

\vspace{3ex}

\noindent\textbf{Outline of the proof of Theorem \ref{Main3}.} By \cite[Theorem 2.3]{ZL11}, if $p_{(1,g-1),(1,g-1)}^{(2,g-2)}=k$, then $\Gamma$ is isomorphic to the digraph in (v). We only need to consider the case that $p_{(1,g-1),(1,g-1)}^{(2,g-2)}=k-1$ or $k-2$.

In Section 2, we give some basic results concerning commutative weakly distance-regular
digraphs which will be used frequently in this paper.

In Section 3, we prove our main result under the assumption that $p_{(1,g-1),(1,g-1)}^{(2,g-2)}=k-1$. We begin with determining the decompositions of $A_{1,g-1}^2$ and $A_{1,g-1}A_{g-1,1}$, respectively. Based on these decompositions, we construct $\Gamma$.

In Section 4, we prove our main result under the assumption that $p_{(1,g-1),(1,g-1)}^{(2,g-2)}=k-2$. We divide our proof into two cases according to whether the set $\Gamma_{1,g-1}^2$ contains $\Gamma_{1,g-1}$. For each case, we compute the number of vertices of $\Gamma$, and determine these digraphs according to \cite{AH} and \cite{AH26}.

In Section 5, we give a proof of Lemma \ref{g=3}.

\section{Preliminaries}

In this section, we always assume that $\Gamma$ is a commutative weakly distance-regular digraph. Now we list some basic properties of intersection numbers.

\begin{lemma}\label{jb}
{\rm (\cite[Proposition 5.1]{ZA99} and \cite[Chapter \Rmnum{2}, Proposition 2.2]{EB84})} For each $\wz{i}:=(a,b)\in\wz{\partial}(\Gamma)$, define $\wz{i}^{*}=(b,a)$. The following hold:
\begin{itemize}
\item [{\rm(i)}] $k_{\wz{d}}k_{\wz{e}}=\sum_{\wz{f}\in\wz{\partial}(\Gamma)}p_{\wz{d},\wz{e}}^{\wz{f}}k_{\wz{f}}$;

\item [{\rm(ii)}] $p_{\wz{d},\wz{e}}^{\wz{f}}k_{\wz{f}}=p_{\wz{f},\wz{e}^{*}}^{\wz{d}}k_{\wz{d}}=p_{\wz{d}^{*},\wz{f}}^{\wz{e}}k_{\wz{e}}$;

\item [{\rm(iii)}] $\sum_{\wz{e}\in\wz{\partial}(\Gamma)}p_{\wz{d},\wz{e}}^{\wz{f}}=k_{\wz{d}}$;

\item [{\rm(iv)}] $\sum_{\wz{f}\in\wz{\partial}(\Gamma)}p_{\wz{d},\wz{e}}^{\wz{f}}p_{\wz{g},\wz{f}}^{\wz{h}}=\sum_{\wz{l}\in\wz{\partial}(\Gamma)}p_{\wz{g},\wz{d}}^{\wz{l}}p_{\wz{l},\wz{e}}^{\wz{h}}$;

\item [{\rm(v)}] ${\rm lcm}(k_{\wz{d}},k_{\wz{e}})\mid p_{\wz{d},\wz{e}}^{\wz{f}}k_{\wz{f}}$;

\item [{\rm(vi)}] $|\Gamma_{\wz{d}}\Gamma_{\wz{e}}|\leq{\rm gcd}(k_{\wz{d}},k_{\wz{e}})$.
\end{itemize}
\end{lemma}

\begin{lemma}\label{jb2}
Suppose that $(1,q-1)\in\wz{\partial}(\Gamma)$ with $q>1$. Then the following conditions are equivalent:
\begin{itemize}
\item [{\rm(i)}] $\Gamma_{1,q-1}\in \Gamma_{1,q-1}^2$;

\item [{\rm(ii)}] $p_{(1,q-1),(1,q-1)}^{(1,q-1)}\neq0$;

\item [{\rm(iii)}] $\{\Gamma_{1,q-1},\Gamma_{q-1,1}\}\subseteq \Gamma_{1,q-1}\Gamma_{q-1,1}$.
\end{itemize}
\end{lemma}
\begin{proof}
By setting $\wz{d}=\wz{e}=\wz{f}=(1,q-1)$ in Lemma \ref{jb} (ii), we get $p_{(1,q-1),(1,q-1)}^{(1,q-1)}=p_{(1,q-1),(q-1,1)}^{(1,q-1)}=p_{(1,q-1),(q-1,1)}^{(q-1,1)}$. The desired result follows.
\end{proof}

\begin{lemma}\label{jb3}
Suppose that each arc of $\Gamma$ is of type $(1,g-1)$. If $a<i$ for each $\Gamma_{a,b}\in\Gamma_{1,g-1}^i$, then $\wz{\partial}(\Gamma)=\{(s,t)\mid\Gamma_{s,t}\in\Gamma_{1,g-1}^j~\textrm{for some $j<i$}\}$.
\end{lemma}
\begin{proof}
For each $\Gamma_{a,b}\in\Gamma_{1,g-1}^i$, since $a<i$, we have $\Gamma_{a,b}\in\Gamma_{1,g-1}^a$, which implies $\Gamma_{a,b}\Gamma_{1,g-1}\subseteq\Gamma_{1,g-1}^{a+1}$, and so $\Gamma_{1,g-1}^{i+1}\subseteq\cup_{j<i}\Gamma_{1,g-1}^j$. By induction, we get $\Gamma_{1,g-1}^{i+l}\subseteq\cup_{j<i}\Gamma_{1,g-1}^j$ for $l>0$. Since $\Gamma$ is strongly connected, the desired result follows.
\end{proof}

Let $P_{\wz{i},\wz{j}}(x,y)=\Gamma_{\wz{i}}(x)\cap\Gamma_{\wz{j}^{*}}(y)$ for all $\wz{i},\wz{j}\in\wz{\partial}(\Gamma)$ and $x,y\in V\Gamma$.

\begin{lemma}\label{jb-1.2}
Suppose that each arc in $\Gamma$ is of type $(1,g-1)$. If $p_{(1,g-1),(1,g-1)}^{(2,g-2)}=k_{1,g-1}-a>a$ for $a\in\{1,2\}$, then $g=3$.
\end{lemma}
\begin{proof}
Let $(x_{0,0},x_{1,0},\ldots,x_{g-1,0})$ be a circuit, where the first subscription of $x$ are read modulo $g$. Without loss of generality, we may assume
\begin{align}
P_{(1,g-1),(1,g-1)}(x_{i-1,0},x_{i+1,0})=\{x_{i,j}\mid 0\leq j\leq l-1\}\label{eq1}
\end{align}
for all $i$, where $l=p_{(1,g-1),(1,g-1)}^{(2,g-2)}$.

Assume the contrary, namely, $g>3$. Since $(x_{i,j},x_{i+1,0},x_{i+2,0},\ldots,x_{i-1,0})$ is a circuit for $0\leq j\leq l-1$, we have $(x_{i-2,0},x_{i,j})\in\Gamma_{2,g-2}$. By \eqref{eq1}, one gets $P_{(1,g-1),(1,g-1)}(x_{i-1,0},x_{i+1,0})\subseteq P_{(2,g-2),(g-1,1)}(x_{i-2,0},x_{i-1,0})$ and $p_{(2,g-2),(g-1,1)}^{(1,g-1)}\geq l$.

\textbf{Case 1.} $p_{(2,g-2),(g-1,1)}^{(1,g-1)}>l$.

By Lemma \ref{jb} (ii), we obtain $p_{(1,g-1),(1,g-1)}^{(2,g-2)}k_{2,g-2}=p_{(2,g-2),(g-1,1)}^{(1,g-1)}k_{1,g-1}$. Since $p_{(2,g-2),(g-1,1)}^{(1,g-1)}>p_{(1,g-1),(1,g-1)}^{(2,g-2)}=k_{1,g-1}-a$ for some $a\in\{1,2\}$, we have $k_{2,g-2}=k_{1,g-1}(k_{1,g-1}-1)/(k_{1,g-1}-2)$ or $k_{2,g-2}=k_{1,g-1}^2/(k_{1,g-1}-a)$. Since ${\rm gcd}(k_{1,g-1},k_{1,g-1}-1)={\rm gcd}(k_{1,g-1}-1,k_{1,g-1}-2)=1$, one gets $a=2$ and $k_{1,g-1}-2\mid k_{1,g-1}^2$. In view of $k_{1,g-1}>4$, we obtain $k_{1,g-1}=6$, which implies $k_{2,g-2}=9$, $p_{(1,g-1),(1,g-1)}^{(2,g-2)}=4$ and $p_{(2,g-2),(g-1,1)}^{(1,g-1)}=6$. By the commutativity of $\Gamma$ and setting $\wz{f}=\wz{d}^*=(1,g-1)$ in Lemma \ref{jb} (iii), one has $\Gamma_{1,g-1}^2=\{\Gamma_{2,g-2}\}$.

Pick distinct vertices $x,y,y'$ such that $x\in P_{(g-1,1),(1,g-1)}(y,y')$. Since $\Gamma_{1,g-1}^2=\{\Gamma_{2,g-2}\}$ with $g>3$, from Lemma \ref{jb2}, we have $\partial(y,y')>1$. Let $\Gamma_{1,g-1}(y)=\{z_i\mid0\leq i\leq5\}$ and $(y,z_0,w)$ be a path such that $\partial(y,y')=\partial(w,y')+2$. Since $p_{(1,g-1),(1,g-1)}^{(2,g-2)}=4$, we may assume $P_{(1,g-1),(1,g-1)}(y,w)=\{z_0,z_1,z_2,z_3\}$. If $(y',z_i)\notin\Gamma_{1,g-1}$ for all $0\leq i\leq 3$, from $\Gamma_{1,g-1}^2=\{\Gamma_{2,g-2}\}$, then $\Gamma_{2,g-2}(x)\subseteq\Gamma_{1,g-1}(y')\cup\{z_0,z_1,z_2,z_3\}$, contrary to the fact that $k_{2,g-2}=9$. Without loss of generality, we may assume $(y',z_0)\in\Gamma_{1,g-1}$. It follows that $(y',w)\in\Gamma_{2,g-2}$, and so $\partial(y,y')=g$. Similarly, $\partial(y',y)=g$. Hence, $\Gamma_{g-1,1}\Gamma_{1,g-1}=\{\Gamma_{0,0},\Gamma_{g,g}\}$.

By setting $\wz{f}=\wz{d}=(1,g-1)$ in Lemma \ref{jb} (iii), one has $p_{(1,g-1),(g,g)}^{(1,g-1)}=5$. Lemma \ref{jb} (iv) implies $p_{(1,g-1),(1,g-1)}^{(2,g-2)}p_{(g-1,1),(2,g-2)}^{(1,g-1)}=6+p_{(g-1,1),(1,g-1)}^{(g,g)}p_{(g,g),(1,g-1)}^{(1,g-1)}$. It follows from the commutativity of $\Gamma$ that $p_{(g-1,1),(1,g-1)}^{(g,g)}=18/5$, a contradiction.

\textbf{Case 2.} $p_{(2,g-2),(g-1,1)}^{(1,g-1)}=l$.

Since $P_{(1,g-1),(1,g-1)}(x_{i-1,0},x_{i+1,0})\subseteq P_{(2,g-2),(g-1,1)}(x_{i-2,0},x_{i-1,0})$ for all $i$, from \eqref{eq1}, we have
\begin{align}
P_{(2,g-2),(g-1,1)}(x_{i-2,0},x_{i-1,0})=\{x_{i,j}\mid 0\leq j\leq l-1\}.\label{eq2}
\end{align}
For $y_i\in P_{(1,g-1),(1,g-1)}(x_{i-1,0},x_{i+1,j})$, since $(y_{i},x_{i+1,j},x_{i+2,0},x_{i+3,0},\ldots,x_{i-1,0})$ is a circuit with distinct vertices $y_{i},x_{i+3,0}$, we get $y_i\in P_{(2,g-2),(g-1,1)}(x_{i-2,0},x_{i-1,0})$, and so $P_{(1,g-1),(1,g-1)}(x_{i-1,0},x_{i+1,j})\subseteq P_{(2,g-2),(g-1,1)}(x_{i-2,0},x_{i-1,0})$ for $0\leq j\leq l-1$. The fact $p_{(2,g-2),(g-1,1)}^{(1,g-1)}=l$ implies that $P_{(1,g-1),(1,g-1)}(x_{i-1,0},x_{i+1,j})=P_{(2,g-2),(g-1,1)}(x_{i-2,0},x_{i-1,0})$. By \eqref{eq2}, one has $(x_{i,j'},x_{i+1,j})\in\Gamma_{1,g-1}$, and so
\begin{align}
P_{(1,g-1),(1,g-1)}(x_{0,0},x_{2,j})&=\{x_{1,j''}\mid0\leq j''\leq l-1\},\label{jb-1.2-eq-1}\\
P_{(2,g-2),(g-1,1)}(x_{0,j},x_{1,j'})&=\{x_{2,j''}\mid0\leq j''\leq l-1\}\label{jb-1.2-eq-2}
\end{align}
for $0\leq j,j'\leq l-1$.

Let $(x_{0,0},x_{1,l},\ldots,x_{g-1,l})$ be a circuit with $x_{1,l}\notin P_{(1,g-1),(1,g-1)}(x_{0,0},x_{2,0})$. By \eqref{jb-1.2-eq-1}, one has $x_{2,l}\notin \{x_{2,j}\mid0\leq j\leq l-1\}$. For each $y\in P_{(1,g-1),(1,g-1)}(x_{0,0},x_{2,l})$, since $(x_{0,0},y,x_{2,l},x_{3,l},\ldots,x_{g-1,l})$ is a circuit, we get $x_{2,l}\in P_{(2,g-2),(g-1,1)}(x_{0,0},y)$. In view of \eqref{jb-1.2-eq-2}, one obtains $P_{(1,g-1),(1,g-1)}(x_{0,0},x_{2,l})\cap P_{(1,g-1),(1,g-1)}(x_{0,0},x_{2,0})=\emptyset$. Repeating this process, we have $l\mid k_{1,g-1}$, which implies $a=2$ and $k_{1,g-1}=4$, a contradiction.
\end{proof}

The commutativity of $\Gamma$ will be used frequently in the sequel, so we no longer refer to it for the sake of simplicity.

\begin{lemma}\label{jb-1.1}
If $2p_{(1,2),(1,2)}^{(2,1)}>k_{1,2}$, then $p_{(1,2),(2,1)}^{\wz{h}}\geq2p_{(1,2),(1,2)}^{(2,1)}-k_{1,2}$ for each $\Gamma_{\wz{h}}\in\Gamma_{1,2}\Gamma_{2,1}$.
\end{lemma}
\begin{proof}
Let $(x,z)\in\Gamma_{\wz{h}}$. Since $p_{(1,2),(2,1)}^{(0.0)}=k_{1,2}$, we may assume $\wz{h}\neq(0,0)$ and $x\neq z$. Pick a vertex $y\in P_{(2,1),(1,2)}(x,z)$. Suppose $P_{(1,2),(1,2)}(x,y)=\{w_i\mid 1\leq i\leq l\}$ with $l=p_{(1,2),(1,2)}^{(2,1)}$. In view of Lemma \ref{jb} (ii), one gets $p_{(1,2),(1,2)}^{(2,1)}k_{1,2}=p_{(2,1),(2,1)}^{(1,2)}k_{1,2}$, which implies that $p_{(2,1),(2,1)}^{(1,2)}=l$. Note that $P_{(1,2),(1,2)}(x,y)\cap P_{(2,1),(2,1)}(y,z)=P_{(1,2),(1,2)}(x,y)\setminus(\dot{\bigcup}_{\wz{i}\neq(2,1)}P_{(2,1),\wz{i}}(y,z)).$
Since $\sum_{\wz{i}\neq(2,1)}p_{(2,1),\wz{i}}^{(1,2)}=k_{1,2}-l$ from Lemma \ref{jb} (iii), we obtain $|P_{(1,2),(1,2)}(x,y)\cap P_{(2,1),(2,1)}(y,z)|\geq l-(k_{1,2}-l)=2l-k_{1,2}.$ By $P_{(1,2),(2,1)}(x,z)\supseteq P_{(1,2),(1,2)}(x,y)\cap P_{(2,1),(2,1)}(y,z)$, we get $p_{(1,2),(2,1)}^{\wz{h}}\geq2l-k_{1,2}$.
\end{proof}

\begin{lemma}\label{(1,2)^2}
Let $\Gamma_{1,2}^2\subseteq\{\Gamma_{2,1},\Gamma_{2,2},\Gamma_{2,3},\Gamma_{2,4}\}$. Then the following holds:
\begin{itemize}
\item[{\rm(i)}] If $\Gamma_{2,2}\in\Gamma_{1,2}^2$ and $(x,z)\in\Gamma_{2,2}$, then $P_{(1,2),(1,2)}(z,x)\subseteq P_{(2,2),(2,1)}(y,z)$ for all $y\in P_{(1,2),(1,2)}(x,z)$.

\item[{\rm (ii)}] If $\Gamma_{2,j}\in\Gamma_{1,2}^2$ and $(x',z')\in\Gamma_{2,j}$ with $j\in\{2,3\}$, then $P_{(1,2),(2,1)}(x',z')\subseteq P_{(2,2),(2,1)}(y',z')\cup P_{(2,3),(2,1)}(y',z')$ for all $y'\in P_{(1,2),(1,2)}(x',z')$;
\end{itemize}
\end{lemma}
\begin{proof}
(i)~For each $w\in P_{(1,2),(1,2)}(z,x)$, since $\Gamma_{1,2}\notin\Gamma_{1,2}^2$, we have $(y,w)\in\Gamma_{2,2}$, and so $P_{(1,2),(1,2)}(z,x)\subseteq\Gamma_{2,2}(y)$ with $y\in P_{(1,2),(1,2)}(x,z)$.

(ii)~For each $w'\in P_{(1,2),(2,1)}(x',z')$, since $\Gamma_{1,2}\notin\Gamma_{1,2}^2$, one gets $(y',w')\in\Gamma_{2,2}\cup\Gamma_{2,3}$, and so $P_{(1,2),(2,1)}(x',z')\subseteq\Gamma_{2,2}(y')\cup\Gamma_{2,3}(y')$ with $y'\in P_{(1,2),(1,2)}(x',z')$.
\end{proof}

\begin{prop}\label{construction}
The digraph ${\rm Cay}(\mathbb{Z}_{3}\times\mathbb{Z}_{k+1},\{(1,1),(1,2),\ldots,(1,k)\})$ is weakly distance-regular for $k\geq4$.
\end{prop}
\begin{proof}
We will show that ${\rm Cay}(\mathbb{Z}_{3}\times\mathbb{Z}_{k+1},\{(1,1),(1,2),\ldots,(1,k)\})$ is weakly distance-transitive. For any vertex $(a,b)$ distinct with $(0,0)$,
we have
\begin{align}\label{distance}
\wz{\partial}((0,0),(a,b))=\left\{
\begin{array}{ll}
(1,2), & \textrm{if}\ a=1,\ b\neq0;\\
(2,1), & \textrm{if}\ a=2,\ b\neq0;\\
(3,3), & \textrm{if}\ a=0,\ b\neq0;\\
(2,4), & \textrm{if}\ a=2,\ b=0;\\
(4,2), & \textrm{if}\ a=1,\ b=0.
\end{array} \right.
\end{align}
Let $(a,b)$ and $(x,y)$ be two vertices satisfying
$\wz{\partial}((0,0),(a,b))=\wz{\partial}((0,0),(x,y))$. It suffices to verify that there exists an automorphism $\sigma$ of the digraph ${\rm Cay}(\mathbb{Z}_{3}\times\mathbb{Z}_{k+1},\{(1,1),(1,2),\ldots,(1,k)\})$ such that
$\sigma(0,0)=(0,0)$ and $\sigma(a,b)=(x,y)$. If $(a,b)=(x,y)$, then the identity permutation is a desired automorphism. Suppose $(a,b)\neq(x,y)$. By \eqref{distance}, one has $b\neq0$, $y\neq0$ and $a=x$. Let $\sigma$ be the permutation on the vertex set of the digraph ${\rm Cay}(\mathbb{Z}_{3}\times\mathbb{Z}_{k+1},\{(1,1),(1,2),\ldots,(1,k)\})$ such that
\begin{align}
\sigma(u,v)=\left\{
\begin{array}{ll}
(u,v), & \textrm{if}\ v\notin\{b,y\};\\
(u,y), & \textrm{if}\ v=b;\\
(u,b), & \textrm{if}\ v=y.
\end{array} \right.\nonumber
\end{align}
Routinely,  $\sigma$ is a desired automorphism.
\end{proof}

\section{$p_{(1,g-1),(1,g-1)}^{(2,g-2)}=k-1$}

In this section, we prove Theorem \ref{Main3} under the assumption that $p_{(1,g-1),(1,g-1)}^{(2,g-2)}=k-1$. Since $k>3$, from Lemma \ref{jb-1.2}, we have $g=3$. Now we prove it step by step.

\begin{step}\label{i=2}
$A_{1,2}^2=(k-1)A_{2,1}+p_{(1,2),(1,2)}^{(2,j)}A_{2,j}$ with $p_{(2,1),(2,1)}^{(1,2)}=k-1$ and $p_{(2,j),(2,1)}^{(1,2)}=1$.
\end{step}

In view of Lemma \ref{jb} (i) and (v), one gets $A_{1,2}^2=(k-1)A_{2,1}+p_{(1,2),(1,2)}^{(i,j)}A_{i,j}$ with $(i,j)\neq(2,1)$. By setting $\wz{d}=\wz{e}=\wz{f}^*=(1,2)$ in Lemma \ref{jb} (ii), we have $p_{(2,1),(2,1)}^{(1,2)}=k-1$. By setting $\wz{d}=\wz{e}=(1,2)$ in Lemma \ref{jb} (i), one gets $p_{(1,2),(1,2)}^{(i,j)}k_{i,j}=k$. Setting $\wz{d}=\wz{e}=(1,2)$ and $\wz{f}=(i,j)$ in Lemma \ref{jb} (ii), we obtain $p_{(i,j),(2,1)}^{(1,2)}=1$.

It suffices to show that $i=2$. Suppose not. Then $(i,j)=(1,2)$. In view of Lemma \ref{jb} (ii) and Lemma \ref{jb-1.1}, we have $p_{(1,2),(2,1)}^{(1,2)}=p_{(1,2),(2,1)}^{(2,1)}\geq k-2$. By setting $\wz{d}=\wz{e}^{*}=(1,2)$ in Lemma \ref{jb} (i), one gets $k^2\geq k+2p_{(1,2),(2,1)}^{(1,2)}k$, which implies $k\leq3$, contrary to the fact that $k>3$.

\begin{step}\label{bkb}
$A_{1,2}A_{2,1}=kI+p_{(1,2),(2,1)}^{(3,3)}A_{3,3}$ with $p_{(1,2),(3,3)}^{(1,2)}=k-1$.
\end{step}

By Lemma \ref{jb} (i) and (iii), it suffices to show that $\Gamma_{1,2}\Gamma_{2,1}=\{\Gamma_{0,0},\Gamma_{3,3}\}$. Suppose not.  By Step \ref{i=2}, we obtain $\Gamma_{1,2}^2=\{\Gamma_{2,1},\Gamma_{2,j}\}$. Since $k>3$, from Lemma \ref{jb2}, we have $\Gamma_{1,2}\Gamma_{2,1}=\{\Gamma_{0,0},\Gamma_{2,j},\Gamma_{j,2}\}$ or $\Gamma_{1,2}\Gamma_{2,1}=\{\Gamma_{0,0},\Gamma_{2,j},\Gamma_{j,2},\Gamma_{3,3}\}$ with $j\in\{2,3\}$. By Lemma \ref{jb-1.1}, we get $p_{(1,2),(2,1)}^{(2,j)}\geq k-2$.  Lemma \ref{(1,2)^2} (ii) implies $p_{(2,j),(2,1)}^{(1,2)}\geq p_{(1,2),(2,1)}^{(2,j)}\geq 2$, contrary to Step \ref{i=2}.

\begin{step}\label{3.2}
$p_{(1,2),(1,2)}^{(2,j)}>1$.
\end{step}

Assume the contrary, namely, $p_{(1,2),(1,2)}^{(2,j)}=1$. By Step \ref{i=2} and Lemma \ref{jb} (i), one has $k_{2,j}=k$. In view of Steps \ref{i=2}, \ref{bkb} and Lemma \ref{jb} (iv), we get $p_{(1,2),(1,2)}^{(2,1)}p_{(2,1),(2,1)}^{(1,2)}+p_{(1,2),(1,2)}^{(2,j)}p_{(2,1),(2,j)}^{(1,2)}=k+p_{(2,1),(1,2)}^{(3,3)}p_{(3,3),(1,2)}^{(1,2)}$, which implies $p_{(1,2),(2,1)}^{(3,3)}=k-2$. By setting $\wz{d}=(1,2)$ and $\wz{e}=(2,1)$ in Lemma \ref{jb} (i), we obtain $k_{3,3}=(k-1)k/(k-2)$. Since $k>3$ and ${\rm gcd}(k-1,k-2)=1$, one has $k=4$ and $k_{3,3}=6$.

We claim that $(x_0,x_3)\notin\Gamma_{3,3}$ for each path $(x_0,x_1,x_2,x_3)$ with $(x_0,x_2),(x_1,x_3)\in\Gamma_{2,j}$. Suppose not. By Step \ref{i=2}, we have $p_{(2,1),(2,1)}^{(1,2)}=3$. Since $p_{(1,2),(2,1)}^{(3,3)}=2$, there exists a vertex $x\in P_{(1,2),(2,1)}(x_0,x_3)$ such that $(x_2,x)\in\Gamma_{2,1}$, which implies $x_1=x$ since $p_{(1,2),(1,2)}^{(2,j)}=1$. Then $(x_1,x_2,x_3)$ is a circuit, contrary to the fact that $(x_1,x_3)\in\Gamma_{2,j}$. Thus, our claim is valid.

Let $(y_0,y_1,y_2,y_3,y_4)$ be a path with $(y_i,y_{i+2})\in\Gamma_{2,j}$ for $0\leq i\leq 2$. By the claim, we may assume $(y_0,y_3)\in\Gamma_{a,b}$ with $(a,b)\neq(3,3)$. For each $y_3'\in\Gamma_{1,2}(y_2)\setminus\{y_3\}$, since $p_{(2,1),(2,1)}^{(1,2)}=3$ from Step \ref{i=2}, one has $(y_1,y_3')\in\Gamma_{2,1}$, which implies $(y_0,y_3')\in\Gamma_{3,3}$ by Step \ref{bkb}. It follows that $\Gamma_{2,j}\Gamma_{1,2}=\{\Gamma_{3,3},\Gamma_{a,b}\}$, $p_{(a,b),(2,1)}^{(2,j)}=1$ and $p_{(3,3),(2,1)}^{(2,j)}=3$. 

In view of Lemma \ref{jb} (ii), we have $p_{(2,j),(1,2)}^{(a,b)}k_{a,b}=p_{(a,b),(2,1)}^{(2,j)}k_{2,j}=4$. By setting $\wz{d}=(2,j)$ and $\wz{e}=(1,2)$ in Lemma \ref{jb} (i), one gets $p_{(2,j),(1,2)}^{(3,3)}=2$. Step \ref{bkb} and Lemma \ref{jb} (iv) imply $p_{(1,2),(2,j)}^{(3,3)}p_{(2,1),(3,3)}^{(2,j)}+p_{(1,2),(2,j)}^{(a,b)}p_{(2,1),(a,b)}^{(2,j)}=k+p_{(2,1),(1,2)}^{(3,3)}p_{(3,3),(2,j)}^{(2,j)}$.
Substituting $p_{(1,2),(2,j)}^{(3,3)},p_{(2,1),(3,3)}^{(2,j)},p_{(2,1),(a,b)}^{(2,j)},p_{(2,1),(1,2)}^{(3,3)}$ into above equation, we obtain $2p_{(3,3),(2,j)}^{(2,j)}=2+p_{(1,2),(2,j)}^{(a,b)}$. Since $p_{(2,j),(1,2)}^{(a,b)}k_{a,b}=4$, we get $p_{(1,2),(2,j)}^{(a,b)}=2$ or $4$.

Since $p_{(2,j),(1,2)}^{(a,b)}=p_{(1,2),(2,j)}^{(a,b)}$, there exists a path $(y_0,y_1',y_2',y_3)$ such that $(y_0,y_2')\in\Gamma_{2,j}$ with $y_2\neq y_2'$. The fact $p_{(2,1),(2,1)}^{(1,2)}=3$ implies $(y_2',y_4)\in\Gamma_{2,1}$. Since $(y_0,y_3)\in\Gamma_{a,b}$ with $(a,b)\neq(3,3)$, from Step \ref{bkb}, we have $y_4\notin P_{(1,3),(3,1)}(y_0,y_3)$, and so $y_4\neq y_1'$, which implies  $(y_4,y_1')\in\Gamma_{3,3}$. By $p_{(1,2),(2,1)}^{(3,3)}=2$, there exists a vertex $y_5\in P_{(1,2),(2,1)}(y_4,y_1')$ such that $(y_5,y_0)\in\Gamma_{1,2}$.

Since $(y_0,y_3)\in\Gamma_{a,b}$ with $(a,b)\neq(3,3)$, from Step \ref{bkb}, we have $(y_3,y_5),(y_4,y_0)\notin\Gamma_{2,1}$. By Step \ref{i=2}, we get $(y_3,y_5),(y_4,y_0)\in\Gamma_{2,j}$. The claim implies $(y_3,y_0)\in\Gamma_{a,b}$, and so $a=b=2$. Then $j=2$. It follows that there exists $y\in P_{(1,2),(1,2)}(y_0,y_3)$ with $y\neq y_2$. Since $p_{(2,1),(2,1)}^{(1,2)}=3$, one has $(y,y_4)\in\Gamma_{2,1}$. The fact $y\in P_{(1,2),(2,1)}(y_4,y_0)$ implies $\Gamma_{2,2}\in\Gamma_{1,2}\Gamma_{2,1}$, contrary to Step \ref{bkb}.

\begin{step}\label{3.3}
$k_{2,j}=1$, $p_{(1,2),(2,1)}^{(3,3)}=k-1$ and $p_{(1,2),(1,2)}^{(2,j)}=k_{3,3}=k$.
\end{step}

In view of Steps \ref{i=2}, \ref{bkb} and Lemma \ref{jb} (iv), one obtains $p_{(1,2),(1,2)}^{(2,1)}p_{(2,1),(2,1)}^{(1,2)}+p_{(1,2),(1,2)}^{(2,j)}p_{(2,1),(2,j)}^{(1,2)}= k+p_{(2,1),(1,2)}^{(3,3)}p_{(3,3),(1,2)}^{(1,2)}$, which implies that  $p_{(1,2),(1,2)}^{(2,j)}=(k-1)p_{(1,2),(2,1)}^{(3,3)}-k^2+3k-1$. Since $p_{(1,2),(1,2)}^{(2,j)}\leq k$, from Lemma \ref{jb-1.1} and Step \ref{3.2}, one has $p_{(1,2),(2,1)}^{(3,3)}=k-1$, and so $p_{(1,2),(1,2)}^{(2,j)}=k$. By Steps \ref{i=2}, \ref{bkb} and Lemma \ref{jb} (i), we obtain $k_{2,j}=1$ and $k_{3,3}=k$. This completes the proof of this step.

\begin{step}\label{digraph}
$\Gamma$ is isomorphic to the digraph in Theorem {\rm\ref{Main3} (iv)}.
\end{step}

Let $(x_{0,0},x_{1,1},x_{2,2},x_{0,0})$ be a circuit and $\Gamma_{1,2}(x_{h,h})=\{x_{h+1,l}\mid 0\leq l\leq k,~l\neq h\}$ for $0\leq h\leq 2$, where the first subscription of $x$ are read modulo $3$. Since $p_{(1,2),(1,2)}^{(2,1)}=k-1$, we may assume $P_{(1,2),(1,2)}(x_{0,0},x_{2,2})=\{x_{1,l}\mid 1\leq l\leq k,~l\neq2\}$, $P_{(1,2),(1,2)}(x_{1,1},x_{0,0})=\{x_{2,l}\mid 2\leq l\leq k\}$ and $P_{(1,2),(1,2)}(x_{2,2},x_{1,1})=\{x_{0,l}\mid 0\leq l\leq k,~l\neq1,2\}$.

Note that $x_{2,2}\notin P_{(2,1),(2,1)}(x_{0,0},x_{1,2})$ and $x_{2,2}\in P_{(1,2),(1,2)}(x_{1,1},x_{0,0})$. Since $p_{(1,2),(1,2)}^{(2,1)}=p_{(2,1),(2,1)}^{(1,2)}=k-1$ from Step \ref{i=2}, there exists $x_{2,1}\in P_{(2,1),(2,1)}(x_{0,0},x_{1,2})$ such that $x_{2,1}\notin P_{(1,2),(1,2)}(x_{1,1},x_{0,0})$. It follows that $\Gamma_{2,1}(x_{0,0})=\{x_{2,l}\mid1\leq l\leq k\}$.

Suppose $P_{(2,1),(2,1)}(x_{0,0},x_{1,s})=P_{(2,1),(2,1)}(x_{0,0},x_{1,t})$, where $s,t$ are two distinct integers with $1\leq s,t\leq k$. Since $p_{(2,1),(2,1)}^{(1,2)}=k-1$, there exists a vertex $x_{2,l}\in\Gamma_{2,1}(x_{0,0})$ such that $x_{2,l}\notin P_{(2,1),(2,1)}(x_{0,0},x_{1,s})$, which implies $P_{(1,2),(1,2)}(x_{0,0},x_{2,l})\subseteq\Gamma_{1,2}(x_{0,0})\setminus\{x_{1,s},x_{1,t}\}$, contrary to the fact that $p_{(1,2),(1,2)}^{(2,1)}=k-1$. Since $\Gamma_{2,1}(x_{0,0})=\{x_{2,l}\mid1\leq l\leq k\}$, we may assume $P_{(2,1),(2,1)}(x_{0,0},x_{1,l})=\{x_{2,s}\mid1\leq s\leq k,~s\neq l\}$ and $(x_{1,l},x_{2,l})\notin\Gamma_{1,2}$ for $1\leq l\leq k$.

Since $(x_{0,0},x_{2,0})\notin\Gamma_{2,1}$, from Step \ref{i=2}, one has $(x_{0,0},x_{2,0})\in\Gamma_{2,j}$. In view of Step \ref{3.3}, we get $p_{(1,2),(1,2)}^{(2,j)}=k$, which implies $(x_{1,l},x_{2,0})\in\Gamma_{1,2}$ for all $1\leq l\leq k$. By Step \ref{i=2}, there exists a vertex $x_{1,0}\in P_{(2,1),(2,j)}(x_{2,2},x_{0,0})$. Hence, $x_{2,h}\in\Gamma_{1,2}(x_{1,0})$ for $1\leq h\leq k$. Since $P_{(2,1),(2,1)}(x_{0,0},x_{1,l'})=\{x_{2,s}\mid1\leq s\leq k,~s\neq l'\}$ for $1\leq l'\leq k$, $(x_{1,l},x_{2,h})$ is an arc for $0\leq l,h\leq k$ with $l\neq h$. Similarly, $(x_{a,l},x_{a+1,h})\in\Gamma_{1,2}$ for all $a$ and $0\leq l,h\leq k$ with $h\neq l$.

By Proposition \ref{construction}, $\Gamma$ is isomorphic to the digraph in Theorem \ref{Main3} (iv).

\section{$p_{(1,g-1),(1,g-1)}^{(2,g-2)}=k-2$}

In this section, we prove Theorem \ref{Main3} under the assumption that $p_{(1,g-1),(1,g-1)}^{(2,g-2)}=k-2$.

\begin{lemma}\label{g=3}
The girth $g$ of $\Gamma$ is $3$.
\end{lemma}
\begin{proof}
The proof is rather long, and we shall prove it in Section 5.
\end{proof}

By setting $\wz{d}=\wz{e}=(1,2)$ and $\wz{f}=(2,1)$ in Lemma \ref{jb} (ii), we have $p_{(2,1),(2,1)}^{(1,2)}=k-2$. Setting $\wz{d}=\wz{e}=\wz{h}=\wz{g}^*=(1,2)$ in Lemma \ref{jb} (iv), one gets
\begin{align}\label{4.1}
k^2-5k+4+\sum_{\wz{i}\neq(2,1)}p_{(1,2),(1,2)}^{\wz{i}}p_{(2,1),\wz{i}}^{(1,2)}=\sum_{\wz{j}\neq(0,0)}p_{(1,2),(2,1)}^{\wz{j}}p_{(1,2),\wz{j}}^{(1,2)}.
\end{align}

\begin{lemma}\label{g=3-1}
If $\Gamma_{1,2}\in\Gamma_{1,2}^2$, then $\Gamma$ is isomorphic to the digraph in Theorem {\rm\ref{Main3} (i)} or {\rm(ii)}.
\end{lemma}
\begin{proof}
In view of Lemma \ref{jb} (ii) and Lemma \ref{jb-1.1}, we get $p_{(1,2),(1,2)}^{(1,2)}=p_{(1,2),(2,1)}^{(1,2)}=p_{(1,2),(2,1)}^{(2,1)}\geq k-4$.

\textbf{Case 1.} $\Gamma_{1,2}^2\neq\{\Gamma_{1,2},\Gamma_{2,1}\}$.

By Lemma \ref{jb} (i) and (v), one has
\begin{align}
A_{1,2}^2=(k-2)A_{2,1}+A_{1,2}+p_{(1,2),(1,2)}^{(2,j)}A_{2,j}\label{eq-case-1}
\end{align}
with $p_{(1,2),(1,2)}^{(2,j)}k_{2,j}=k$ and $j>1$. Since $p_{(1,2),(1,2)}^{(1,2)}\geq k-4$, we get $k=4$ or $5$. Lemma \ref{jb} (ii) implies $p_{(2,1),(2,j)}^{(1,2)}k=p_{(1,2),(1,2)}^{(2,j)}k_{2,j}$, and so $p_{(2,1),(2,j)}^{(1,2)}=1$.

Let $(x,y,z,w)$ be a path such that $z,w\in\Gamma_{1,2}(x)$. Since $p_{(1,2),(2,1)}^{(1,2)}=p_{(1,2),(1,2)}^{(1,2)}=1$, from \eqref{eq-case-1}, we have $(y,w)\in\Gamma_{2,1}\cup\Gamma_{2,j}$. Suppose $(y,w)\in\Gamma_{2,1}$. Note that $z\in P_{(1,2),(2,1)}(x,y)$, $w\in P_{(1,2),(1,2)}(x,y)$ and $w\in P_{(1,2),(2,1)}(x,z)$. Again, since $p_{(1,2),(2,1)}^{(1,2)}=p_{(1,2),(1,2)}^{(1,2)}=1$, one gets $y,z,w\notin\Gamma_{1,2}(u)\cup\Gamma_{2,1}(u)$ for each $u\in\Gamma_{1,2}(x)\setminus\{y,z,w\}$. It follows from Lemma \ref{jb} (iii) that $k=\sum_{\wz{l}}p_{(1,2),\wz{l}}^{(1,2)}\geq p_{(1,2),(2,1)}^{(1,2)}+p_{(1,2),(1,2)}^{(1,2)}+p_{(1,2),(0,0)}^{(1,2)}+3$, contrary to the fact that $k\leq5$. Then $(y,w)\in\Gamma_{2,j}$.

Since $x\in P_{(2,1),(1,2)}(y,w)$, one has $\Gamma_{2,j}\in\Gamma_{1,2}\Gamma_{2,1}$. Suppose that $\Gamma_{1,2}\Gamma_{2,1}\neq\{\Gamma_{0,0},\Gamma_{1,2},\Gamma_{2,1},\Gamma_{2,j}\}$. Then $|\Gamma_{1,2}\Gamma_{2,1}|\geq5$. By Lemma \ref{jb} (i), (v) and (vi), we have $k=5$ and $A_{1,2}A_{2,1}=5I+A_{1,2}+A_{2,1}+p_{(1,2),(2,1)}^{(2,j)}A_{2,j}+p_{(1,2),(2,1)}^{\wz{i}}A_{\wz{i}}$ with $p_{(1,2),(2,1)}^{\wz{i}}k_{\wz{i}}=p_{(1,2),(2,1)}^{(2,j)}k_{2,j}=5$ and $\wz{i}\notin\{(1,2),(2,1),(2,j)\}$. It follows from Lemma \ref{jb} (ii) that $p_{(1,2),\wz{i}}^{(1,2)}=p_{(1,2),(2,j)}^{(1,2)}=1$. In view of \eqref{4.1} and \eqref{eq-case-1}, one has $p_{(1,2),(1,2)}^{(2,j)}+3=p_{(1,2),(2,1)}^{\wz{i}}+p_{(1,2),(2,1)}^{(2,j)},$ contrary to the fact that $p_{(1,2),(2,1)}^{\wz{i}},p_{(1,2),(2,1)}^{(2,j)},p_{(1,2),(1,2)}^{(2,j)}\in\{1,5\}$.

Note that $\Gamma_{1,2}\Gamma_{2,1}=\{\Gamma_{0,0},\Gamma_{1,2},\Gamma_{2,1},\Gamma_{2,j}\}$. Then $j=2$. By Lemma \ref{jb} (i), we have $A_{1,2}A_{2,1}=kI+A_{1,2}+A_{2,1}+p_{(1,2),(2,1)}^{(2,2)}A_{2,2}$ with $p_{(1,2),(2,1)}^{(2,2)}k_{2,2}=k^2-3k$. It follows from Lemma \ref{jb} (ii) that $p_{(1,2),(2,2)}^{(1,2)}=p_{(1,2),(2,1)}^{(2,2)}k_{2,2}/k=k-3$. By \eqref{4.1} and \eqref{eq-case-1}, we obtain $k^2-5k+3+p_{(1,2),(1,2)}^{(2,2)}=(k-3)p_{(1,2),(2,1)}^{(2,2)}$. If $k=4$, then $p_{(1,2),(1,2)}^{(2,2)}=1+p_{(1,2),(2,1)}^{(2,2)}$, contrary to the fact that $p_{(1,2),(2,1)}^{(2,2)}k_{2,2}=p_{(1,2),(1,2)}^{(2,2)}k_{2,2}=4$. Then $k=5$ and $3+p_{(1,2),(1,2)}^{(2,2)}=2p_{(1,2),(2,1)}^{(2,2)}$. Since $p_{(1,2),(1,2)}^{(2,2)}k_{2,2}=5$ and $p_{(1,2),(2,1)}^{(2,2)}k_{2,2}=10$, from Lemma \ref{jb} (iii), we get $p_{(1,2),(1,2)}^{(2,2)}=1$, $p_{(1,2),(2,1)}^{(2,2)}=2$ and $k_{2,2}=5$.

Pick a vertex $u_1\in P_{(1,2),(2,1)}(x,w)$. Note that $p_{(1,2),(1,2)}^{(1,2)}=p_{(1,2),(2,1)}^{(1,2)}=1$. Since $w\in P_{(1,2),(2,1)}(x,z)$ and $y\in P_{(1,2),(1,2)}(x,z)$, we have $(z,u_1)\notin\Gamma_{1,2}\cup\Gamma_{2,1}$, and so $(z,u_1)\in\Gamma_{2,2}$ from \eqref{eq-case-1}. The fact $\Gamma_{1,2}\Gamma_{2,1}=\{\Gamma_{0,0},\Gamma_{1,2},\Gamma_{2,1},\Gamma_{2,2}\}$ and $x\in P_{(2,1),(1,2)}(y,u_1)$ imply $(y,u_1)\in\Gamma_{1,2}\cup\Gamma_{2,1}\cup\Gamma_{2,2}$. Let $\Gamma_{1,2}(w)=\{u_1,u_2,u_3,u_4,u_5\}$. Since $p_{(2,2),(2,1)}^{(1,2)}=1$, one gets $(z,u_i)\in\Gamma_{1,2}\cup\Gamma_{2,1}$ for $2\leq i\leq 5$. By \eqref{eq-case-1}, we obtain $(y,u_i)\in\Gamma_{1,2}\cup\Gamma_{2,1}\cup\Gamma_{2,2}$ for $1\leq i\leq 5$. It follows that $\Gamma_{2,2}\Gamma_{1,2}\subseteq\{\Gamma_{1,2},\Gamma_{2,1},\Gamma_{2,2}\}$, and so $\Gamma_{1,2}^3\subseteq\{\Gamma_{0,0},\Gamma_{1,2},\Gamma_{2,1},\Gamma_{2,2}\}$. By Lemma \ref{jb3}, one has $\wz{\partial}(\Gamma)=\{(0,0),(1,2),(2,1),(2,2)\}$ with $|V\Gamma|=16$, contrary to \cite{AH}.

\textbf{Case 2.} $\Gamma_{1,2}^2=\{\Gamma_{1,2},\Gamma_{2,1}\}$.

By setting $\wz{d}=\wz{e}=(1,2)$ in Lemma \ref{jb} (i), we have $p_{(1,2),(1,2)}^{(1,2)}=2$. In view of Lemma \ref{jb} (ii), one gets $p_{(1,2),(2,1)}^{(1,2)}=p_{(1,2),(2,1)}^{(2,1)}=2$, which implies $k\leq 6$ from Lemma \ref{jb-1.1}. By setting $\wz{d}=\wz{f}=(1,2)$ in Lemma \ref{jb} (iii), we obtain $k=5$ or $6$.

Suppose $k=5$. By setting $\wz{d}=\wz{e}^*=(1,2)$ in Lemma \ref{jb} (i), one has $A_{1,2}A_{2,1}=5I+2A_{1,2}+2A_{2,1}$, and so $\Gamma_{1,2}^{3}\subseteq\{\Gamma_{0,0},\Gamma_{1,2},\Gamma_{2,1}\}$. Lemma \ref{jb3} implies $\wz{\partial}(\Gamma)=\{(0,0),(1,2),(2,1)\}$ with $|V\Gamma|=11$. In view of \cite{AH}, $\Gamma$ is isomorphic to the digraph in Theorem \ref{Main3} (i).

Suppose $k=6$. By Lemma \ref{jb} (v) and setting $\wz{d}=\wz{e}^*=(1,2)$ in Lemma \ref{jb} (i), we have $A_{1,2}A_{2,1}=6I+2A_{1,2}+2A_{2,1}+p_{(1,2),(2,1)}^{(3,3)}A_{3,3}$ with $p_{(1,2),(2,1)}^{(3,3)}k_{3,3}=6$. Hence, $\Gamma_{1,3}^3=\{\Gamma_{0,0},\Gamma_{1,2},\Gamma_{2,1},\Gamma_{3,3}\}$. In view of Lemma \ref{jb} (ii), we get $p_{(1,2),(2,1)}^{(3,3)}k_{3,3}=p_{(1,2),(3,3)}^{(1,2)}k$, and so $p_{(1,2),(3,3)}^{(1,2)}=1$. It follows from \eqref{4.1} and \eqref{eq-case-1} that $p_{(1,2),(2,1)}^{(3,3)}=6$, and so $k_{3,3}=1$. Then $\Gamma_{3,3}\Gamma_{1,2}=\{\Gamma_{2,1}\}$, and so $\Gamma_{1,2}^4=\{\Gamma_{0,0},\Gamma_{1,2},\Gamma_{2,1},\Gamma_{3,3}\}$. By Lemma \ref{jb3}, we obtain $\wz{\partial}(\Gamma)=\{(0,0),(1,2),(2,1),(3,3)\}$ with $|V\Gamma|=14$. It follows from \cite{AH} that $\Gamma$ is isomorphic to the digraph in Theorem \ref{Main3} (ii).
\end{proof}

In the following, we only need to consider the case that $\Gamma_{1,2}\notin\Gamma_{1,2}^2$.

\begin{lemma}\label{g=3-2}
If $\Gamma_{1,2}\notin\Gamma_{1,2}^2$, then $\Gamma_{1,2}\Gamma_{2,1}\neq\{\Gamma_{0,0},\Gamma_{3,3}\}$.
\end{lemma}
\begin{proof}
Assume the contrary, namely, $A_{1,2}A_{2,1}=kI+p_{(1,2),(2,1)}^{(3,3)}A_{3,3}$. It follows from Lemma \ref{jb} (iii) that $p_{(1,2),(3,3)}^{(1,2)}=k-1$. By Lemma \ref{jb-1.1}, we have $k-4\leq p_{(1,2),(2,1)}^{(3,3)}\leq k$. Lemma \ref{jb} (iii) and Lemma \ref{jb2} imply $\sum_{i>1}p_{(2,1),(2,i)}^{(1,2)}=2$. By \eqref{4.1}, we obtain $k^2-5k+4+2k\geq (k-1)p_{(1,2),(2,1)}^{(3,3)}$. Since $k>3$, one has $p_{(1,2),(2,1)}^{(3,3)}\in\{k-4,k-3,k-2\}$. If $p_{(1,2),(1,2)}^{(2,i)}=1$ for all $\Gamma_{2,i}\in\Gamma_{1,2}^2\setminus\{\Gamma_{2,1}\}$, from \eqref{4.1}, then $k^2-5k+4+2=(k-1)p_{(1,2),(2,1)}^{(3,3)}$, a contradiction. Then $p_{(1,2),(1,2)}^{(2,i)}>1$ for some $i>1$.

\textbf{Case 1.} $\Gamma_{2,4}\notin\Gamma_{1,2}^2$.

Let $(x_0,x_1,\ldots,x_{i+1})$ be a circuit with $(x_0,x_2)\in\Gamma_{2,i}$ and $i=\max\{a\mid\Gamma_{2,a}\in\Gamma_{1,2}^2\}$, where the subscription of $x$ could be read modulo $i+2$.

Suppose $\Gamma_{1,2}^2=\{\Gamma_{2,1},\Gamma_{2,2},\Gamma_{2,3}\}$. Then $i=3$. Since $\sum_{i>1}p_{(2,1),(2,i)}^{(1,2)}=2$, from Lemma \ref{(1,2)^2} (i), we have $p_{(2,1),(2,2)}^{(1,2)}=p_{(2,1),(2,3)}^{(1,2)}=p_{(1,2),(1,2)}^{(2,2)}=1$, and so $p_{(1,2),(1,2)}^{(2,3)}>1$. Since $p_{(1,2),(2,1)}^{(3,3)}\in\{k-2,k-3,k-4\}$, from \eqref{4.1}, one gets $p_{(1,2),(2,1)}^{(3,3)}=k-3$ and $p_{(1,2),(1,2)}^{(2,3)}=k-2$. By Lemma \ref{jb} (ii), we obtain $k_{2,3}p_{(1,2),(1,2)}^{(2,3)}=kp_{(2,1),(2,3)}^{(1,2)}=k$, which implies $k=4$. It follows that there exists $x_1'\in P_{(1,2),(1,2)}(x_0,x_2)$ with $x_1\neq x_1'$. The fact $\partial(x_3,x_0)\leq2$ and $\Gamma_{1,2}\Gamma_{2,1}=\{\Gamma_{0,0},\Gamma_{3,3}\}$ imply $(x_1,x_3),(x_1',x_3)\notin\Gamma_{2,1}$. Since $p_{(2,1),(2,2)}^{(1,2)}=p_{(2,1),(2,3)}^{(1,2)}=1$, we may assume $(x_{1},x_3)\in\Gamma_{2,3}$ and $(x_1',x_3)\in\Gamma_{2,2}$. The fact $p_{(1,2),(1,2)}^{(2,3)}=2$ implies that there exists $x_2'\in P_{(1,2),(1,2)}(x_1,x_3)$ with $x_2'\neq x_2$. Since  $\Gamma_{1,2}\notin\Gamma_{1,2}^2$, we obtain $x_4\notin P_{(1,2),(1,2)}(x_3,x_1')$, and so $(x_2,x_4)\notin\Gamma_{2,2}$ from Lemma \ref{(1,2)^2} (i). The fact $\Gamma_{1,2}\Gamma_{2,1}=\{\Gamma_{0,0},\Gamma_{3,3}\}$ implies $(x_2,x_4)\in\Gamma_{2,3}$ and $(x_2',x_4)\in\Gamma_{2,2}$. Then there exists $x_3'\in P_{(1,2),(1,2)}(x_2,x_4)$ with $x_3\neq x_3'$. The fact $\Gamma_{1,2}\notin\Gamma_{1,2}^2$ implies $x_0\notin P_{(1,2),(1,2)}(x_4,x_2')$. By Lemma \ref{(1,2)^2} (i), we get $(x_3,x_0)\notin\Gamma_{2,2}$. The fact $\Gamma_{1,2}\Gamma_{2,1}=\{\Gamma_{0,0},\Gamma_{3,3}\}$ and $p_{(2,1),(2,2)}^{(1,2)}=p_{(2,1),(2,3)}^{(1,2)}=1$ imply $(x_3,x_0)\in\Gamma_{2,3}$, $(x_3',x_0)\in\Gamma_{2,2}$, and $x_1$ or $x_1'\in\Gamma_{2,2}(x_4)$. By Lemma \ref{(1,2)^2} (i), one obtains $x_1$ or $x_1'\in P_{(1,2),(1,2)}(x_0,x_3')$, contrary to the fact that $\Gamma_{1,2}\notin\Gamma_{1,2}^2$.

Suppose $\Gamma_{1,2}^2=\{\Gamma_{2,1},\Gamma_{2,i}\}$ for some $i\in\{2,3\}$. Note that $p_{(2,1),(2,i)}^{(1,2)}=2$. By Lemma \ref{jb} (ii), we have $k_{2,i}p_{(1,2),(1,2)}^{(2,i)}=kp_{(2,1),(2,i)}^{(1,2)}=2k$, and so $p_{(1,2),(1,2)}^{(2,i)}\mid2k$. Since $p_{(1,2),(2,1)}^{(3,3)}\in\{k-2,k-3,k-4\}$ with $k>3$, from \eqref{4.1}, one gets $p_{(1,2),(1,2)}^{(2,i)}=2$ and $k=k_{2,i}=5$.  The fact $\Gamma_{1,2}\Gamma_{2,1}=\{\Gamma_{0,0},\Gamma_{3,3}\}$ implies $(x_{1},x_3)\in\Gamma_{2,i}$. Similarly, $(x_l,x_{l+2})\in\Gamma_{2,i}$ for all $l$. Let $x_l'\in P_{(1,2),(1,2)}(x_{l-1},x_{l+1})$ with $x_l\neq x_l'$. By similar argument, one has $x_{l+1},x_{l+1}'\in P_{(2,i),(2,1)}(x_{l-1},x_l)$. Since $p_{(1,2),(1,2)}^{(2,i)}=2$, there exists $x_{l+1}''\in P_{(1,2),(1,2)}(x_{l},x_{l+2}')$ with $x_{l+1}''\neq x_{l+1}$. By similar argument, we get $x_{l+1},x_{l+1}',x_{l+1}''\in P_{(2,i),(2,1)}(x_{l-1},x_l)$. The fact $p_{(2,1),(2,i)}^{(1,2)}=2$ implies $x_{l+1}'=x_{l+1}''$, and so $P_{(2,i),(2,1)}(x_0,x_1)=P_{(2,i),(2,1)}(x_0,x_1')=\{x_2,x_2'\}$.

Let $(x_0=y_0,y_1,\ldots,y_{i+1})$ be a circuit with $(y_0,y_2)\in\Gamma_{2,i}$ and $y_2\notin\{x_2,x_2'\}$. Note that $y_1\notin\{x_1,x_1'\}$. Similarly, there exist $y_1'\in P_{(1,2),(1,2)}(y_0,y_2)$ with $y_1'\neq y_1$ and $y_2'\in P_{(1,2),(1,2)}(y_1,y_3)$ with $y_2'\neq y_2$ such that $(y_1',y_2')$ is an arc and $P_{(2,i),(2,1)}(x_0,y_1)=P_{(2,i),(2,1)}(x_0,y_1')=\{y_2,y_2'\}$. It follows that $|\{x_1,x_1',y_1,y_1'\}|=4$. Note that $k_{2,i}=5$. Let $(x_0=z_0,z_1,\ldots,z_{i+1})$ be a circuit with $(z_0,z_2)\in\Gamma_{2,i}$ and $z_2\notin\{x_2,x_2',y_2,y_2'\}$. Similarly, there exists a vertex $z_1'\in P_{(1,2),(1,2)}(z_0,z_2)$ with $z_1'\neq z_1$ such that $z_2\in P_{(2,i),(2,1)}(x_0,z_1)=P_{(2,i),(2,1)}(x_0,z_1')$. It follows that $x_1,x_1',y_1,y_1',z_1,z_1'$ are pairwise distinct, contrary to the fact that $k=5$.

\textbf{Case 2.} $\Gamma_{2,4}\in\Gamma_{1,2}^2$.

Since $p_{(2,1),(2,1)}^{(1,2)}=k-2$, from Lemma \ref{jb} (iii), we get $\sum_{i>1}p_{(2,i),(2,1)}^{(1,2)}=2$. Let $(x,y,z)$ be a path such that $(x,z)\in\Gamma_{2,4}$.

Suppose $\Gamma_{1,2}^2=\{\Gamma_{2,1},\Gamma_{2,4}\}$. Note that $p_{(2,4),(2,1)}^{(1,2)}=2$. By Lemma \ref{jb} (ii), we have $k_{2,4}p_{(1,2),(1,2)}^{(2,4)}=kp_{(2,1),(2,4)}^{(1,2)}=2k$. Since $p_{(1,2),(3,3)}^{(1,2)}=k-1$ and $p_{(1,2),(2,1)}^{(3,3)}\in\{k-2,k-3,k-4\}$ with $k>3$, from \eqref{4.1}, one gets $p_{(1,2),(1,2)}^{(2,4)}=2$, $k=k_{2,4}=5$ and $p_{(1,2),(2,1)}^{(3,3)}=2$. By setting $\wz{d}=\wz{e}^*=(1,2)$ in Lemma \ref{jb} (i), we obtain $k_{3,3}=10$. Let $P_{(1,2),(1,2)}(x,z)=\{y,y'\}$ and $\Gamma_{1,2}(z)=\{w_0,w_1,w_2,w_3,w_4\}$. Since $p_{(2,1),(2,1)}^{(1,2)}=3$, we may assume $w_0,w_1,w_2\in\Gamma_{2,1}(y)$ and $w_0\in\Gamma_{2,1}(y)$. Note that $\Gamma_{1,2}\Gamma_{2,1}=\{\Gamma_{0,0},\Gamma_{3,3}\}$ and $p_{(2,1),(1,2)}^{(3,3)}=2$. Since $x,w_0\in P_{(2,1),(1,2)}(y,y')$, one has $w_1,w_2\notin\Gamma_{2,1}(y')$, which implies $w_3,w_4\in\Gamma_{2,1}(y')$. It follows that $(x,w_i)\in\Gamma_{3,3}$ for $0\leq i\leq 4$, and so $\Gamma_{2,4}\Gamma_{1,2}=\{\Gamma_{3,3}\}$. By Lemma \ref{jb} (i), we get $k_{2,4}k=p_{(2,4),(1,2)}^{(3,3)}k_{3,3}$, a contradiction.

Since $\sum_{i>1}p_{(2,i),(2,1)}^{(1,2)}=2$, we have $\Gamma_{1,2}^2=\{\Gamma_{2,1},\Gamma_{2,i},\Gamma_{2,4}\}$ and $p_{(2,i),(2,1)}^{(1,2)}=p_{(2,4),(2,1)}^{(1,2)}=1$ for some $i\in\{2,3\}$. Let $(x_0,x_1,\ldots,x_{i+1})$ be a circuit with $(x_0,x_2)\in\Gamma_{2,i}$. For each $x_1'\in P_{(1,2),(1,2)}(x_0,x_2)$, since $\Gamma_{1,2}\Gamma_{2,1}=\{\Gamma_{0,0},\Gamma_{3,3}\}$, one gets $x_1,x_1'\in P_{(2,1),(2,i)}(x_2,x_3)$, and so $x_1=x_1'$ from $p_{(2,i),(2,1)}^{(1,2)}=1$. It follows that $p_{(1,2),(1,2)}^{(2,i)}=1$. Since $p_{(1,2),(3,3)}^{(1,2)}=k-1$, from \eqref{4.1}, one gets $k^2-5k+5+p_{(1,2),(1,2)}^{(2,4)}=(k-1)p_{(1,2),(2,1)}^{(3,3)}$. Since $p_{(1,2),(2,1)}^{(3,3)}\in\{k-4,k-3,k-2\}$ and $p_{(1,2),(1,2)}^{(2,4)}k_{2,4}=p_{(2,4),(2,1)}^{(1,2)}k=k$ from Lemma \ref{jb} (ii), we obtain $k=4$, $k_{2,4}=2$, $p_{(1,2),(1,2)}^{(2,4)}=2$ and $p_{(1,2),(2,1)}^{(3,3)}=1$. By setting $\wz{d}=\wz{e}^*=(1,2)$ in Lemma \ref{jb} (i), we obtain $k_{3,3}=12$. For $w\in P_{(2,1),(2,1)}(y,z)$, since $\Gamma_{1,2}\Gamma_{2,1}=\{\Gamma_{0,0},\Gamma_{3,3}\}$, one has $(x,w)\in\Gamma_{3,3}$, which implies $p_{(3,3),(2,1)}^{(2,4)}\neq0$. Lemma \ref{jb} (ii) implies $12p_{(2,4),(1,2)}^{(3,3)}=p_{(2,4),(1,2)}^{(3,3)}k_{3,3}=p_{(3,3),(2,1)}^{(2,4)}k_{2,4}\leq 8$, a contradiction.
\end{proof}

Since $\Gamma_{1,2}\notin\Gamma_{1,2}^2$, from Lemmas \ref{jb2} and \ref{g=3-2}, one has $\Gamma_{2,i}\in\Gamma_{1,2}^2\cap\Gamma_{1,2}\Gamma_{2,1}$ for some $i\in\{2,3\}$. In view of Lemma \ref{(1,2)^2} (ii), we get $p_{(2,2),(2,1)}^{(1,2)}+p_{(2,3),(2,1)}^{(1,2)}\geq p_{(1,2),(2,1)}^{(2,i)}$. Since $p_{(2,1),(2,1)}^{(1,2)}=k-2$, from Lemma \ref{jb} (iii) and Lemma \ref{jb-1.1}, we obtain $2=\sum_{i>1}p_{(2,i),(2,1)}^{(1,2)}\geq p_{(1,2),(2,1)}^{(2,i)}\geq k-4$. Then $k\in\{4,5,6\}$.

\begin{lemma}\label{g=3-3}
If $\Gamma_{1,2}\notin\Gamma_{1,2}^2$, then $\Gamma_{2,2}\notin\Gamma_{1,2}\Gamma_{2,1}$.
\end{lemma}
\begin{proof}
Assume the contrary, namely, $\Gamma_{2,2}\in\Gamma_{1,2}\Gamma_{2,1}$. Since $\sum_{i>1}p_{(2,i),(2,1)}^{(1,2)}=2$, by Lemma \ref{(1,2)^2}, we get $p_{(1,2),(1,2)}^{(2,2)}=p_{(1,2),(2,1)}^{(2,2)}=1$. Lemma \ref{jb-1.1} implies $k=4$ or $5$. Let $(x,y,z)$ be a path such that $(x,z)\in\Gamma_{2,2}$. Pick a vertex $y'\in P_{(2,1),(1,2)}(x,z)$.

Suppose $|P_{(2,1),(1,2)}(y,y')|=1$ and $k=5$. Let $\Gamma_{1,2}(z)=\{w_i\mid0\leq i\leq4\}$. Since $p_{(2,1),(1,2)}^{(1,2)}=3$, we may assume $w_0,w_1,w_2\in\Gamma_{2,1}(y)$ and $w_0\in\Gamma_{2,1}(y')$. The fact $w_0\in P_{(2,1),(1,2)}(y,y')$ implies $w_1,w_2\notin\Gamma_{2,1}(y')$, and so $w_3,w_4\in\Gamma_{2,1}(y')$. Since $p_{(1,2),(1,2)}^{(2,2)}=1$, we have $w_i\in P_{(1,2),(1,2)}(z,x)$ for some $i\in\{0,1,2,3,4\}$, which implies $y\in P_{(1,2),(2,1)}(w_i,x)$ or $y'\in P_{(1,2),(1,2)}(w_i,x)$, contrary to the fact that $\Gamma_{1,2}\notin\Gamma_{1,2}^2$.

Suppose $k=5$. Then $|P_{(2,1),(1,2)}(y,y')|>1$. Since $p_{(1,2),(2,1)}^{(2,2)}=1$ and $x\in P_{(1,2),(1,2)}(y',y)$, one has $(y',y)\in\Gamma_{2,3}$. It follows that $\Gamma_{2,3}\in\Gamma_{1,2}^2\cap\Gamma_{1,2}\Gamma_{2,1}$. The fact $\sum_{i>1}p_{(2,i),(2,1)}^{(1,2)}=2$ implies $p_{(2,3),(2,1)}^{(1,2)}=p_{(2,2),(2,1)}^{(1,2)}=1$. In view of Lemma \ref{jb} (ii), we obtain $k_{2,3}p_{(1,2),(1,2)}^{(2,3)}=kp_{(2,3),(2,1)}^{(1,2)}=5$. Since $p_{(1,2),(2,1)}^{(2,3)}>1$, from Lemma \ref{jb} (iii), one gets $p_{(1,2),(1,2)}^{(2,3)}=1$ and $k_{2,3}=5$. By Lemma \ref{jb} (ii), we have $p_{(1,2),(2,3)}^{(1,2)}k=p_{(1,2),(2,1)}^{(2,3)}k_{2,3}$, and so $p_{(1,2),(2,3)}^{(1,2)}=p_{(1,2),(3,2)}^{(1,2)}\geq2$, which implies $p_{(1,2),(0,0)}^{(1,2)}+p_{(1,2),(2,2)}^{(1,2)}+p_{(1,2),(2,3)}^{(1,2)}+p_{(1,2),(3,2)}^{(1,2)}\geq6>k$, contrary to Lemma \ref{jb} (iii). Hence, $k=4$.

By Lemma \ref{jb} (iii), we have $\sum_{\wz{j}\neq(0,0)}p_{(1,2),\wz{j}}^{(1,2)}=3$. Since $\sum_{i>1}p_{(2,i),(2,1)}^{(1,2)}=2$, from \eqref{4.1}, one gets $p_{(1,2),(1,2)}^{(2,i)}>1$ for some $i\in\{3,4\}$, which implies $p_{(2,i),(2,1)}^{(1,2)}=p_{(2,2),(2,1)}^{(1,2)}=1$. Pick a vertex $w\in P_{(1,2),(1,2)}(z,x)$. Since $\Gamma_{1,2}\notin\Gamma_{1,2}^2$, we obtain $y\in P_{(3,1),(2,2)}(z,w)$. The fact $x\in P_{(1,2),(2,1)}(y',w)$ implies $y'\in P_{(3,1),(2,i)}(z,w)$, $i=3$ and $\Gamma_{2,3},\Gamma_{3,2}\in\Gamma_{1,2}\Gamma_{2,1}$. Since $\sum_{\wz{j}\neq(0,0)}p_{(1,2),\wz{j}}^{(1,2)}=3$, we get $\Gamma_{1,2}\Gamma_{2,1}=\{\Gamma_{0,0},\Gamma_{2,3},\Gamma_{3,2},\Gamma_{2,2}\}$ with $p_{(2,2),(1,2)}^{(1,2)}=p_{(2,3),(1,2)}^{(1,2)}=p_{(3,2),(1,2)}^{(1,2)}=1$. By Lemma \ref{jb} (ii), one obtains $p_{(1,2),(2,1)}^{(2,3)}k_{2,3}=p_{(2,3),(1,2)}^{(1,2)}k=4$ and $p_{(1,2),(1,2)}^{(2,3)}k_{2,3}=p_{(2,3),(2,1)}^{(1,2)}k=4$, which imply $p_{(1,2),(1,2)}^{(2,3)}=p_{(1,2),(2,1)}^{(2,3)}=p_{(1,2),(2,1)}^{(3,2)}$. In view of \eqref{4.1}, we have $1+p_{(1,2),(1,2)}^{(2,3)}=1+2p_{(1,2),(2,1)}^{(2,3)}$, a contradiction.
\end{proof}

\begin{lemma}\label{g=3-4}
If $\Gamma_{1,2}\notin\Gamma_{1,2}^2$, then $\Gamma_{2,2}\notin\Gamma_{1,2}^2$.
\end{lemma}
\begin{proof}
Assume the contrary, namely, $\Gamma_{2,2}\in\Gamma_{1,2}^2$. By Lemmas \ref{g=3-2} and \ref{g=3-3}, we have $\Gamma_{1,2}^2\cap\Gamma_{1,2}\Gamma_{2,1}=\{\Gamma_{2,3}\}$. Since $\sum_{i>1}p_{(2,i),(2,1)}^{(1,2)}=2$, we have $p_{(2,2),(2,1)}^{(1,2)}=p_{(2,3),(2,1)}^{(1,2)}=1$, and so $\Gamma_{1,2}^2=\{\Gamma_{2,1},\Gamma_{2,2},\Gamma_{2,3}\}$. Let $(y,z,w)$ be a path such that $(y,w)\in\Gamma_{2,3}$. Pick a vertex $u\in P_{(1,2),(2,1)}(y,w)$.

Suppose that there exists $z'\in P_{(1,2),(1,2)}(y,w)$ with $z'\neq z$ or $u'\in P_{(1,2),(2,1)}(y,w)$ with $u'\neq u$. Note that $y\in P_{(2,1),(1,2)}(z,u)\cap P_{(2,1),(1,2)}(z',u)$ and $w\in P_{(1,2),(1,2)}(z,u)\cap P_{(1,2),(1,2)}(z',u)$, or $y\in P_{(2,1),(1,2)}(z,u)\cap P_{(2,1),(1,2)}(z,u')$ and $w\in P_{(1,2),(1,2)}(z,u)\cap P_{(1,2),(1,2)}(z,u')$. Since $\Gamma_{2,1},\Gamma_{2,2}\notin\Gamma_{1,2}\Gamma_{2,1}$, we have $z,z'\in P_{(2,1),(2,3)}(w,u)$ or $u,u'\in P_{(2,3),(2,1)}(z,w)$, contrary to the fact that $p_{(2,1),(2,3)}^{(1,2)}=1$.

Note that $p_{(1,2),(1,2)}^{(2,3)}=p_{(1,2),(2,1)}^{(2,3)}=1$. Lemma \ref{(1,2)^2} (i) implies $p_{(1,2),(1,2)}^{(2,2)}=1$. By Lemma \ref{jb} (v) and setting $\wz{d}=\wz{e}=(1,2)$ in Lemma \ref{jb} (i), we obtain $k_{2,2}=k_{2,3}=k$. In view of Lemma \ref{jb} (ii), one gets $p_{(1,2),(2,3)}^{(1,2)}k=p_{(1,2),(2,1)}^{(2,3)}k_{2,3}$, and so $p_{(1,2),(2,3)}^{(1,2)}=p_{(1,2),(3,2)}^{(1,2)}=1$. Since $k\in\{4,5,6\}$ and $\Gamma_{2,2}\notin\Gamma_{1,2}\Gamma_{2,1}$, by setting $\wz{d}=\wz{f}=(1,2)$ in Lemma \ref{jb} (iii), we have $\Gamma_{1,2}\Gamma_{2,1}=\{\Gamma_{0,0},\Gamma_{2,3},\Gamma_{3,2},\Gamma_{3,3}\}$ and $p_{(1,2),(3,3)}^{(1,2)}=k-3$. By \eqref{4.1}, one obtains $k=5$ and $p_{(1,2),(2,1)}^{(3,3)}=2$. In view of Lemma \ref{jb} (ii), one has $p_{(1,2),(2,1)}^{(3,3)}k_{3,3}=p_{(1,2),(3,3)}^{(1,2)}k$, and so $k_{3,3}=5$.

Pick a vertex $x\in P_{(2,1),(2,2)}(y,z)$. Since $y\in P_{(2,1),(1,2)}(z,u)$ and $\Gamma_{2,1},\Gamma_{2,2}\notin\Gamma_{1,2}\Gamma_{2,1}$, we have $(z,u)\in\Gamma_{2,3}$. Since $p_{(2,1),(2,1)}^{(1,2)}=3$, we may assume $P_{(2,1),(2,1)}(u,y)=\{v_0,v_1,v_2\}$ and $(v_2,z)\in\Gamma_{2,1}$. The fact $p_{(1,2),(2,1)}^{(2,3)}=1$ and $v_2\in P_{(1,2),(2,1)}(z,u)$ imply that $v_0,v_1\notin\Gamma_{1,2}(z)$. Since $\sum_{i>1}p_{(2,i),(2,1)}^{(1,2)}=2$, one gets $x\in\{v_0,v_1\}$, and so $(x,w)\in\Gamma_{2,2}\cup\Gamma_{3,2}$ from $P_{(1,2),(2,1)}(x,z)=\emptyset$. If $(x,w)\in\Gamma_{3,2}$, then there exists $u'\in P_{(2,1),(1,2)}(x,w)$, which implies $u',z\in P_{(2,1),(2,3)}(w,u)$ since $x\in P_{(1,2),(2,1)}(u',u)$, contrary to the fact that $p_{(2,3),(2,1)}^{(1,2)}=1$. Thus, $(x,w)\in\Gamma_{2,2}$.

Let $\Gamma_{1,2}(z)=\{w,w_1,w_2,w_3,w_4\}$ and $\Gamma_{2,1}(y)=\{x,x_1,x_2,x_3,x_4\}$. In view of $p_{(1,2),(1,2)}^{(2,2)}=p_{(2,1),(1,2)}^{(2,3)}=1$, we may assume $(w_1,x),(x_1,w)\in\Gamma_{1,2}$. The fact $\Gamma_{1,2}\Gamma_{2,1}=\{\Gamma_{0,0},\Gamma_{2,3},\Gamma_{3,2},\Gamma_{3,3}\}$ implies $w_1,w\notin\Gamma_{2,1}(y)$ and $x,x_1\notin\Gamma_{1,2}(z)$. Since $p_{(2,1),(2,1)}^{(1,2)}=3$, we get $w_2,w_3,w_4\in\Gamma_{2,1}(y)$ and $x_2,x_3,x_4\in\Gamma_{1,2}(z)$. Then $(x,w_i),(x_i,w)\in\Gamma_{2,3}\cup\Gamma_{3,2}\cup\Gamma_{3,3}$ for $2\leq i\leq4$. It follows that $\Gamma_{2,2}\Gamma_{1,2},\Gamma_{2,3}\Gamma_{1,2}\subseteq\{\Gamma_{1,2},\Gamma_{2,1},\Gamma_{2,2},\Gamma_{2,3},\Gamma_{3,2},\Gamma_{3,3}\}$. Hence, $\Gamma_{1,2}^3\subseteq\{\Gamma_{0,0}, \Gamma_{1,2}, \Gamma_{2,1},\Gamma_{2,2},\Gamma_{2,3},\Gamma_{3,2},\Gamma_{3,3}\}$.

Let $(y_0,y_2)\in\Gamma_{3,3}$ and $(y_2,y_3)\in\Gamma_{1,2}$. Since $p_{(2,3),(2,1)}^{(1,2)}=1$ and $p_{(2,1),(1,2)}^{(3,3)}=2$, there exists $y_1\in P_{(2,1),(1,2)}(y_0,y_2)$ with $(y_1,y_3)\in\Gamma_{2,1}\cup\Gamma_{2,2}$. It follows that $(y_0,y_3)\in\Gamma_{1,2}\cup\Gamma_{2,1}\cup\Gamma_{2,2}\cup\Gamma_{2,3}\cup\Gamma_{3,2}\cup\Gamma_{3,3}$, and so $\Gamma_{3,3}\Gamma_{1,2}\subseteq\{\Gamma_{1,2},\Gamma_{2,1},\Gamma_{2,2},\Gamma_{2,3},\Gamma_{3,2},\Gamma_{3,3}\}$.

Let $(z_0,z_2)\in\Gamma_{3,2}$. Since $p_{(1,2),(2,1)}^{(2,3)}=1$, there exists $z_1\in P_{(2,1),(1,2)}(z_0,z_2)$. Let $\Gamma_{1,2}(z_2)=\{u_0,u_1,u_2,u_3,u_4\}$. Without loss of generality, we may assume $u_0\in P_{(1,2),(1,2)}(z_2,z_0)$. Since $z_0\in P_{(1,2),(2,1)}(z_1,u_0)$, we have $(z_1,u_0)\in\Gamma_{2,3}$. The fact $p_{(2,3),(2,1)}^{(1,2)}=1$ implies $(z_1,u_i)\in\Gamma_{2,1}\cup\Gamma_{2,2}$ for $1\leq i\leq 4$. It follows that $(u_i,z_0)\in\Gamma_{1,2}\cup\Gamma_{2,1}\cup\Gamma_{2,2}\cup\Gamma_{2,3}\cup\Gamma_{3,2}\cup\Gamma_{3,3}$. Hence, $\Gamma_{3,2}\Gamma_{1,2}\subseteq\{\Gamma_{1,2},\Gamma_{2,1},\Gamma_{2,2},\Gamma_{2,3},\Gamma_{3,2},\Gamma_{3,3}\}$, and so $\Gamma_{1,2}^4\subseteq\{\Gamma_{0,0}, \Gamma_{1,2}, \Gamma_{2,1},\Gamma_{2,2},\Gamma_{2,3},\Gamma_{3,2},\Gamma_{3,3}\}$.

Since $k=k_{2,3}=k_{2,2}=k_{3,3}=5$, from Lemma \ref{jb3}, one obtains $\wz{\partial}(\Gamma)=\{(0,0),(1,2),(2,1),(2,2),(2,3),(3,2),(3,3)\}$ with $|V\Gamma|=31$, contrary to \cite{AH31}.
\end{proof}

\begin{lemma}\label{g=3-5}
If $\Gamma_{1,2}\notin\Gamma_{1,2}^2$, then $\Gamma_{1,2}^2=\{\Gamma_{2,1},\Gamma_{2,3}\}$.
\end{lemma}
\begin{proof}
Assume the contrary, namely, $\Gamma_{1,2}^2\neq\{\Gamma_{2,1},\Gamma_{2,3}\}$. Since $\Gamma_{1,2}\notin\Gamma_{1,2}^2$, from Lemmas \ref{g=3-2} and \ref{g=3-4}, we have $\Gamma_{1,2}^2=\{\Gamma_{2,1},\Gamma_{2,3},\Gamma_{2,4}\}$. The fact $\sum_{i>1}p_{(2,i),(2,1)}^{(1,2)}=2$ implies $p_{(2,3),(2,1)}^{(1,2)}=p_{(2,4),(2,1)}^{(1,2)}=1$. By Lemma \ref{(1,2)^2} (ii), we obtain $p_{(1,2),(2,1)}^{(2,3)}=p_{(1,2),(2,1)}^{(3,2)}=1$. Lemma \ref{jb-1.1} implies $k\in\{4,5\}$. Let $(x_0,x_1,x_2)$ be a path such that $(x_0,x_2)\in\Gamma_{2,3}$. Pick a vertex $u\in P_{(2,1),(1,2)}(x_0,x_2)$. The fact $\Gamma_{1,2}^2\cap\Gamma_{1,2}\Gamma_{2,1}=\{\Gamma_{2,3}\}$ implies $(u,x_1)\in\Gamma_{2,3}$.

Suppose $k=5$. Let $\Gamma_{1,2}(x_2)=\{v_i\mid0\leq i\leq 4\}$. Since $p_{(2,1),(2,1)}^{(1,2)}=3$, we may assume $v_0,v_1,v_2\in\Gamma_{2,1}(x_1)$ and $v_0\in\Gamma_{2,1}(u)$. The fact $p_{(2,1),(1,2)}^{(2,3)}=1$ implies $v_1,v_2\notin\Gamma_{2,1}(u)$, and so $v_3,v_4\in\Gamma_{2,1}(u)$. It follows that $P_{(2,4),(2,1)}(x_1,x_2)=\emptyset$, a contradiction. Then $k=4$.

Since $p_{(1,2),(2,3)}^{(1,2)}=p_{(1,2),(3,2)}^{(1,2)}$, from Lemma \ref{jb} (iii), we get $p_{(1,2),(2,3)}^{(1,2)}=p_{(1,2),(3,2)}^{(1,2)}=p_{(1,2),(3,3)}^{(1,2)}=1$, which implies $\Gamma_{1,2}\Gamma_{2,1}=\{\Gamma_{0,0},\Gamma_{2,3},\Gamma_{3,2},\Gamma_{3,3}\}$. In view of Lemma \ref{jb} (i) and (v), we obtain $k_{2,3}=4$ and $p_{(1,2),(1,2)}^{(2,4)}k_{2,4}=p_{(1,2),(2,1)}^{(3,3)}k_{3,3}=4$. By \eqref{4.1}, one has $p_{(1,2),(1,2)}^{(2,4)}=1+p_{(1,2),(2,1)}^{(3,3)}$, which implies $p_{(1,2),(2,1)}^{(3,3)}=1$ and $p_{(1,2),(1,2)}^{(2,4)}=2$. It follows that $k_{3,3}=4$ and $k_{2,4}=2$.

Let $(y_0,y_2)\in\Gamma_{2,4}$ and $P_{(1,2),(1,2)}(y_0,y_2)=\{y_1,y_1'\}$. We claim that $(y_1,y_3)\in\Gamma_{2,1}$ if $(y_1',y_3)\notin\Gamma_{2,1}$ for $y_3\in\Gamma_{1,2}(y_2)$. Since $p_{(2,1),(2,1)}^{(1,2)}=2$, we may assume $y_3',y_3''\in P_{(2,1),(2,1)}(y_1',y_2)$. The fact $y_0\in P_{(2,1),(1,2)}(y_1,y_1')$ and $p_{(1,2),(2,1)}^{\wz{l}}\leq1$ for $\wz{l}\neq(0,0)$ imply $y_3',y_3''\notin\Gamma_{2,1}(y_1)$, and so $(y_1,y_3)\in\Gamma_{2,1}$. Thus, our claim is valid.

Since $p_{(1,2),(2,1)}^{(2,3)}=1$, there exists a vertex $x_3\in P_{(1,2),(2,1)}(x_0,x_2)$. The fact $x_0\in P_{(2,1),(1,2)}(x_1,x_3)$ and $\Gamma_{1,2}^2\cap\Gamma_{1,2}\Gamma_{2,1}=\{\Gamma_{2,3}\}$ imply $(x_1,x_3)\in\Gamma_{2,3}$. Since $p_{(1,2),(3,2)}^{(1,2)}=1$, there exist $x_5\in P_{(1,2),(3,2)}(x_0,x_3)$ and $x_4\in P_{(1,2),(1,2)}(x_3,x_5)$ such that $x_5\notin\{x_1,x_3\}$ and $(x_0,x_4)\in\Gamma_{2,3}$. It follows that there also exist $x_7\in P_{(1,2),(3,2)}(x_0,x_5)$ and $x_6\in P_{(1,2),(1,2)}(x_5,x_7)$ such that $x_7\notin\{x_3,x_5\}$ and $(x_0,x_6)\in\Gamma_{2,3}$.

Suppose $x_1=x_7$. Let $\Gamma_{1,2}(x_0)=\{x_1,x_3,x_5,x\}$. Since $p_{(1,2),(2,3)}^{(1,2)}=1$, we have $x_i\in P_{(1,2),(2,3)}(x_0,x)$ for some $i\in\{1,3,5\}$, which implies $x,x_{i+2}\in P_{(1,2),(3,2)}(x_0,x_i)$, contrary to the fact that $p_{(1,2),(3,2)}^{(1,2)}=1$. Hence, $\Gamma_{1,2}(x_0)=\{x_1,x_3,x_5,x_7\}$.

Since $p_{(1,2),(3,2)}^{(1,2)}=p_{(1,2),(2,3)}^{(1,2)}=1$, we have $x_1\in P_{(1,2),(3,2)}(x_0,x_7)$, which implies that there exists $x_8\in P_{(1,2),(1,2)}(x_7,x_1)$ such that $(x_0,x_8)\in\Gamma_{2,3}$. The fact $x_0\in P_{(2,1),(2,3)}(x_5,x_6)$ implies $(x_4,x_6)\in\Gamma_{2,1}\cup\Gamma_{2,4}$.

Suppose $(x_4,x_6)\in\Gamma_{2,4}$. Since $p_{(1,2),(1,2)}^{(2,4)}=2$, there exists $x_5'\in P_{(1,2),(1,2)}(x_4,x_6)$ with $x_5'\neq x_5$. By the claim, $(x_7,x_5',x_3)$ is a path. Note that $x_7\in P_{(1,2),(1,2)}(x_0,x_5')$ and $x_3\in P_{(1,2),(2,1)}(x_0,x_5')$. Since $\Gamma_{1,2}^2\cap\Gamma_{1,2}\Gamma_{2,1}=\{\Gamma_{2,3}\}$, one gets $(x_0,x_5')\in\Gamma_{2,3}$. The fact $x_8,x_5'\in P_{(2,3),(2,1)}(x_0,x_7)$ and $p_{(2,3),(2,1)}^{(1,2)}=1$ imply $x_8=x_5'$. Since $x_0,x_8\in P_{(2,1),(1,2)}(x_1,x_3)$ and $p_{(1,2),(2,1)}^{(2,3)}=1$, one gets $x_8=x_0$, a contradiction.

Note that $(x_4,x_6)\in\Gamma_{2,1}$. Similarly, $(x_2,x_4),(x_6,x_8),(x_8,x_2)\in\Gamma_{2,1}$. Since $\Gamma_{1,2}^2=\{\Gamma_{2,1},\Gamma_{2,3},\Gamma_{2,4}\}$, one has $(x_8,x_4)\in\Gamma_{2,1}$, a contradiction.
\end{proof}

\begin{lemma}\label{g=3-6}
If $\Gamma_{1,2}\notin\Gamma_{1,2}^2$, then $k=6$.
\end{lemma}
\begin{proof}
Assume the contrary, namely, $k\in\{4,5\}$. By Lemmas \ref{g=3-2}--\ref{g=3-5}, we have $\Gamma_{1,2}^2=\{\Gamma_{2,1},\Gamma_{2,3}\}$ and $\Gamma_{2,3},\Gamma_{3,2}\in\Gamma_{1,2}\Gamma_{2,1}$. Since $\sum_{i>1}p_{(2,i),(2,1)}^{(1,2)}=2$, one gets $p_{(2,1),(2,3)}^{(1,2)}=2$. By setting $\wz{d}=\wz{e}=(1,2)$ in Lemma \ref{jb} (i), one has $p_{(1,2),(1,2)}^{(2,3)}k_{2,3}=2k$.

Suppose $p_{(1,2),(1,2)}^{(2,3)}=1$. Since $k\in\{4,5\}$ and $k_{2,3}=2k$, by setting $\wz{d}=\wz{e}^*=(1,2)$ in Lemma \ref{jb} (i), we get $k=5$, $\Gamma_{1,2}\Gamma_{2,1}=\{\Gamma_{0,0},\Gamma_{2,3},\Gamma_{3,2}\}$ and $p_{(1,2),(2,1)}^{(2,3)}=p_{(1,2),(2,1)}^{(3,2)}=1$. It follows from Lemma \ref{jb} (ii) that $p_{(1,2),(2,1)}^{(2,3)}k_{2,3}=p_{(1,2),(2,3)}^{(1,2)}k$, and so $p_{(1,2),(2,3)}^{(1,2)}=p_{(1,2),(3,2)}^{(1,2)}=2$, contrary to \eqref{4.1}. Hence,  $p_{(1,2),(1,2)}^{(2,3)}>1$.

Since $p_{(1,2),(2,1)}^{(2,3)}\neq0$, from Lemma \ref{jb} (iii), we have $p_{(1,2),(1,2)}^{(2,3)}<k$. The fact $k\in\{4,5\}$ and $p_{(1,2),(1,2)}^{(2,3)}k_{2,3}=2k$ imply that $p_{(1,2),(1,2)}^{(2,3)}=2$ and $k_{2,3}=k$.

Let $(x,z)\in\Gamma_{2,3}$ and $w\in P_{(1,2),(2,1)}(x,z)$. Since $p_{(1,2),(1,2)}^{(2,3)}=2$, there exist distinct vertices $y,y'\in P_{(1,2),(1,2)}(x,z)$. The fact $\Gamma_{1,2}^2\cap\Gamma_{1,2}\Gamma_{2,1}=\{\Gamma_{2,3}\}$ implies $y,y'\in\Gamma_{3,2}(w)$. It follows that $p_{(1,2),(2,3)}^{(1,2)}=p_{(1,2),(3,2)}^{(1,2)}\geq2$. By Lemma \ref{jb} (iii), we have $k\geq p_{(1,2),(2,3)}^{(1,2)}+p_{(1,2),(3,2)}^{(1,2)}+p_{(1,2),(0,0)}^{(1,2)}$, which implies $k=5$, $p_{(1,2),(2,3)}^{(1,2)}=p_{(1,2),(3,2)}^{(1,2)}=2$ and $\Gamma_{1,2}\Gamma_{2,1}=\{\Gamma_{0,0},\Gamma_{2,3},\Gamma_{3,2}\}$. In view of Lemma \ref{jb} (ii), one gets $p_{(1,2),(2,1)}^{(2,3)}k_{2,3}=p_{(1,2),(2,3)}^{(1,2)}k$, and so $p_{(1,2),(2,1)}^{(2,3)}=p_{(1,2),(2,1)}^{(3,2)}=2$.

Pick a vertex $w'\in P_{(1,2),(2,1)}(x,z)$ with $w\neq w'$. Since $\Gamma_{1,2}^2\cap\Gamma_{1,2}\Gamma_{2,1}=\{\Gamma_{2,3}\}$, one has $y,y'\in\Gamma_{3,2}(w')$. Since $\Gamma_{1,2}\Gamma_{2,1}=\{\Gamma_{0,0},\Gamma_{2,3},\Gamma_{3,2}\}$, we may assume $(y,y')\in\Gamma_{2,3}$. Note that $y',w,w'\in P_{(1,2),(3,2)}(x,y)$, contrary to the fact that $p_{(1,2),(3,2)}^{(1,2)}=2$.
\end{proof}

\begin{lemma}\label{g=3-7}
If $\Gamma_{1,2}\notin\Gamma_{1,2}^2$, then $\Gamma$ is isomorphic to the digraph in Theorem {\rm\ref{Main3} (iii)}.
\end{lemma}
\begin{proof}
By Lemma \ref{g=3-6}, one obtains $k=6$. In view of Lemma \ref{g=3-5}, one gets $\Gamma_{1,2}^2=\{\Gamma_{2,1},\Gamma_{2,3}\}$. Since $p_{(1,2),(2,3)}^{(1,2)}=p_{(1,2),(3,2)}^{(1,2)}$, from Lemma \ref{g=3-2} and Lemma \ref{jb} (iii), we obtain $\Gamma_{1,2}\Gamma_{2,1}=\{\Gamma_{0,0},\Gamma_{2,3},\Gamma_{3,2},\Gamma_{3,3}\}$. Since $\sum_{i>1}p_{(2,i),(2,1)}^{(1,2)}=2$, one has $p_{(2,3),(2,1)}^{(1,2)}=2$. By Lemma \ref{jb-1.1} and Lemma \ref{(1,2)^2} (ii), we get $p_{(1,2),(2,1)}^{(2,3)}=2$.

Let $(x_0,x_1,x_2)$ be a path such that $(x_0,x_2)\in\Gamma_{2,3}$. Since $P_{(1,2),(2,1)}(x_0,x_2)\subseteq P_{(2,3),(2,1)}(x_1,x_2)$ from Lemma \ref{(1,2)^2} (ii), we have $p_{(1,2),(3,2)}^{(1,2)}\geq2$. By setting $\wz{d}=\wz{f}=(1,2)$ in Lemma \ref{jb} (iii), one gets $p_{(1,2),(2,3)}^{(1,2)}=p_{(1,2),(3,2)}^{(1,2)}=2$ and $p_{(1,2),(3,3)}^{(1,2)}=1$.
In view of Lemma \ref{jb} (ii), we have $p_{(1,2),(2,1)}^{(2,3)}k_{2,3}=p_{(1,2),(2,3)}^{(1,2)}k$, and so $k_{2,3}=6$. By setting $\wz{d}=\wz{e}=(1,2)$ in Lemma \ref{jb} (i), we have $p_{(1,2),(1,2)}^{(2,3)}=2$. It follows from \eqref{4.1} that $p_{(1,2),(2,1)}^{(3,3)}=6$, and so $\Gamma_{3,3}\Gamma_{1,2}=\{\Gamma_{2,1}\}$. By setting $\wz{d}=\wz{e}^*=(1,2)$ in Lemma \ref{jb} (i), one gets $k_{3,3}=1$.

Since $p_{(1,2),(1,2)}^{(2,3)}=p_{(1,2),(2,1)}^{(2,3)}=2$, there exist three distinct vertices $z_0,z_1\in P_{(1,2),(2,1)}(x_0,x_2)$ and $x_1'\in P_{(1,2),(1,2)}(x_0,x_2)$ with $x_1\neq x_1'$.

Since $p_{(2,1),(2,1)}^{(1,2)}=4$ and $\Gamma_{1,2}\notin\Gamma_{1,2}^2$, there exist distinct vertices $y_0,y_1,y_2,y_3\in P_{(2,1),(2,1)}(x_1,x_2)$ such that $\Gamma_{1,2}(x_2)=\{y_0,y_1,y_2,y_3,z_0,z_1\}$. If $(x_0,y_i)\in\Gamma_{3,3}$ for some $i\in\{0,1,2,3\}$, from $p_{(1,2),(2,1)}^{(3,3)}=6$, then $z_1\in P_{(1,2),(2,1)}(x_0,y_i)$, contrary to the fact that $\Gamma_{1,2}\notin\Gamma_{1,2}^2$. Since $p_{(1,2),(2,3)}^{(1,2)}=p_{(1,2),(3,2)}^{(1,2)}=2$, we may assume $y_0,y_1\in\Gamma_{2,3}(x_0)$ and $y_2,y_3\in\Gamma_{3,2}(x_0)$. It follows that $\Gamma_{2,3}\Gamma_{1,2}=\{\Gamma_{1,2},\Gamma_{2,3},\Gamma_{3,2}\}$ and $\Gamma_{1,2}^3=\{\Gamma_{0,0},\Gamma_{1,2},\Gamma_{2,3},\Gamma_{3,2},\Gamma_{3,3}\}$.

Since $y_0,y_1\in P_{(2,3),(2,1)}(x_0,x_2)$, we have $p_{(2,3),(2,1)}^{(2,3)}\geq2$. By setting $\wz{d}=\wz{f}=(2,3)$ and $\wz{e}=(2,1)$ in Lemma \ref{jb} (ii), one gets $p_{(1,2),(2,3)}^{(2,3)}=p_{(2,3),(2,1)}^{(2,3)}\geq2$. It follows that there exist $z_2,z_3\in P_{(1,2),(2,3)}(x_0,x_2)$ such that $\Gamma_{1,2}(x_0)=\{x_1,x_1',z_0,z_1,z_2,z_3\}$. Then $\Gamma_{3,2}\Gamma_{1,2}=\{\Gamma_{1,2},\Gamma_{2,1},\Gamma_{3,2}\}$ and $\Gamma_{1,2}^4=\{\Gamma_{1,2},\Gamma_{2,1},\Gamma_{2,3},\Gamma_{3,2}\}$. By Lemma \ref{jb3}, we obtain $\wz{\partial}(\Gamma)=\{(0,0),(1,2),(2,1),(2,3),(3,2),(3,3)\}$ with $|V\Gamma|=26$. In view of \cite{AH26}, $\Gamma$ is isomorphic to the digraph in Theorem \ref{Main3} (iii).
\end{proof}

Combining Step \ref{digraph} in Section 3 and Lemmas \ref{g=3-1}, \ref{g=3-7}, we complete the proof of Theorem \ref{Main3}.

\section{Proof of Lemma \ref{g=3}}

Before giving a proof of Lemma \ref{g=3}, we need some more basic terminologies and notations. Let $\Gamma$ be a weakly distance-regular digraphs. For a nonempty subset $E$ of $R$, we say $E$ {\em closed} if $\Gamma_{\wz{i}^{*}}\Gamma_{\wz{j}}\subseteq E$ for any $\Gamma_{\wz{i}}$ and $\Gamma_{\wz{j}}$ in $E$, and denote $\langle E\rangle$ the minimum closed subset containing $E$. For any nonempty closed subset $F$ of $R$, the \emph{quotient digraph} of $\Gamma$ over $F$, denoted by $\Gamma/F$, is defined as the digraph with vertex set $V\Gamma/F$ in which $(F(x),F(y))$ is an arc whenever there is an arc in $\Gamma$ from $F(x)$ to $F(y)$, where $F(x)=\{y\in V\Gamma\mid(x,y)\in \cup_{f\in F}f\}$ and $V\Gamma/F=\{F(x)\mid x\in V\Gamma\}$.

In this section, we prove Lemma \ref{g=3} by contradiction. Suppose $g>3$. By Lemma \ref{jb-1.2}, we have $k=4$. We give a proof of Lemma \ref{g=3} step by step.

\begin{step}\label{step1}
For some $\Gamma_{\wz{h}}\in\Gamma_{1,g-1}^2$, $p_{(1,g-1),(1,g-1)}^{\wz{h}}\neq2$.
\end{step}
Suppose, to the contrary that $p_{(1,g-1),(1,g-1)}^{\wz{h}}=2$ for all $\Gamma_{\wz{h}}\in\Gamma_{1,g-1}^2$. By setting $\wz{d}=\wz{e}=\wz{h}=\wz{g}^*=(1,g-1)$ in Lemma \ref{jb} (iv), we have
$2\sum_{\wz{f}}p_{(g-1,1),\wz{f}}^{(1,g-1)}=4+\sum_{\wz{l}\neq(0,0)}p_{(g-1,1),(1,g-1)}^{\wz{l}}p_{\wz{l},(1,g-1)}^{(1,g-1)}.$ In view of Lemma \ref{jb} (iii), one gets $\sum_{\wz{f}}p_{(g-1,1),\wz{f}}^{(1,g-1)}=\sum_{\wz{l}}p_{\wz{l},(1,g-1)}^{(1,g-1)}=4$, which implies $\sum_{\wz{l}\neq(0,0)}p_{(g-1,1),(1,g-1)}^{\wz{l}}p_{\wz{l},(1,g-1)}^{(1,g-1)}=4$. It follows that $p_{(1,g-1),(g-1,1)}^{(l,l)}=2$, $p_{(l,l),(1,g-1)}^{(1,g-1)}=1$ and $p_{(1,g-1),(g-1,1)}^{\wz{j}}=1$ for all $\Gamma_{\wz{j}}\in\Gamma_{1,g-1}\Gamma_{g-1,1}\setminus\{\Gamma_{0,0},\Gamma_{l,l}\}$. Denote $F=\langle\Gamma_{l,l}\rangle$. Note that $\{\Gamma_{0,0},\Gamma_{l,l}\}\subseteq F\cap\Gamma_{1,g-1}\Gamma_{g-1,1}$.

\textbf{Case 1.} $\{\Gamma_{0,0},\Gamma_{l,l}\}\subsetneq F\cap\Gamma_{1,g-1}\Gamma_{g-1,1}$.

Let $s$ be the minimum positive integer such that $\Gamma_{\wz{j}}\in\Gamma_{l,l}^{s}$ for some $\Gamma_{\wz{j}}\in\Gamma_{1,g-1}\Gamma_{g-1,1}\setminus\{\Gamma_{0,0},\Gamma_{l,l}\}$. Pick vertices $x_0,x_1,\ldots,x_{s}$ such that $(x_0,x_s)\in\Gamma_{\wz{j}}$ and $(x_i,x_{i+1})\in\Gamma_{l,l}$ for $0\leq i\leq s-1$. Since $p_{(1,g-1),(g-1,1)}^{\wz{j}}=1$ and $p_{(1,g-1),(g-1,1)}^{(l,l)}=2$, there exist five vertices $w,y_0,y_1,y_{2},y_{3}$ such that $w\in P_{(g-1,1),(1,g-1)}(x_0,x_s)$, $P_{(1,g-1),(g-1,1)}(x_0,x_1)=\{y_0,y_1\}$ and $P_{(1,g-1),(g-1,1)}(x_{s-1},x_s)=\{y_{2},y_{3}\}$.

If $(x_s,y_0)\in\Gamma_{1,g-1}$, from $p_{(l,l),(1,g-1)}^{(1,g-1)}=1$, then $P_{(l,l),(1,g-1)}(x_1,y_0)=\{x_0\}$, which implies that $(x_1,x_s)\in\Gamma_{\wz{a}}$ for some $\Gamma_{\wz{a}}\in\Gamma_{1,g-1}\Gamma_{g-1,1}$ with $\wz{a}\notin\{(0,0),(l,l)\}$ since $y_0\in P_{(1,g-1),(g-1,1)}(x_1,x_s)$, contrary to the minimality of $s$. Similarly, $\Gamma_{1,g-1}\cap\{(x_0,y_{2}),(x_0,y_{3}),(x_s,y_{0}),(x_s,y_{1})\}=\emptyset$.

Since $p_{(1,g-1),(1,g-1)}^{\wz{h}}=2$ for all $\Gamma_{\wz{h}}\in\Gamma_{1,g-1}^2$, there exist $z\in P_{(1,g-1),(1,g-1)}(w,y_0)$ and $z'\in P_{(1,g-1),(1,g-1)}(w,y_1)$ with $z,z'\notin\{x_0,x_{s-1},x_s\}$. If $z=x_1$ or $z'=x_1$, from $p_{(l,l),(1,g-1)}^{(1,g-1)}=1$, then $P_{(1,g-1),(l,l)}(w,x_1)=\{x_0\}$, and so $(x_1,x_s)\in\Gamma_{\wz{a}}$ for some $\Gamma_{\wz{a}}\in\Gamma_{1,g-1}\Gamma_{g-1,1}$ with $\wz{a}\notin\{(0,0),(l,l)\}$ since $w\in P_{(g-1,1),(1,g-1)}(x_1,x_s)$, contrary to the minimality of $s$. Hence, $x_1\notin\{z,z'\}$. The fact $x_0,x_1\in P_{(g-1,1),(1,g-1)}(y_0,y_1)$ and $p_{(1,g-1),(g-1,1)}^{\wz{j}}\leq2$ for all $\wz{j}\neq(0,0)$ imply that $z\neq z'$, and so $\Gamma_{1,g-1}(w)=\{z,z',x_0,x_s\}$. Note that $x_{s-1},x_s\in P_{(g-1,1),(1,g-1)}(y_2,y_3)$. Since $(x_0,y_{2}),(x_0,y_{3})\notin\Gamma_{1,g-1}$ and $p_{(1,g-1),(1,g-1)}^{\wz{h}}=2$ for all $\Gamma_{\wz{h}}\in\Gamma_{1,g-1}^2$, we may assume $(z,y_{2}),(z',y_{3})\in\Gamma_{1,g-1}$. Since $p_{(1,g-1),(l,l)}^{(1,g-1)}=1$, we obtain $z$ or $z'\in P_{(1,g-1),(l,l)}(w,x_s)$. It follows that $\{z,x_{s-1}\}\subseteq P_{(l,l),(1,g-1)}(x_s,y_{2})$ or $\{z',x_{s-1}\}\subseteq P_{(l,l),(1,g-1)}(x_s,y_{3})$, a contradiction.

\textbf{Case 2.} $\{\Gamma_{0,0},\Gamma_{l,l}\}=F \cap\Gamma_{1,g-1}\Gamma_{g-1,1}$.

Since $k=4$ and $p_{(1,g-1),(l,l)}^{(1,g-1)}=1$, the valency of the quotient digraph $\Gamma/F$ is $2$. Let $g'$ be the girth of the quotient digraph $\Gamma/F$. Observe that $g'\leq g$.

\textbf{Case 2.1.} $g'<g$.

Let $s$ be the minimum integer such that $\Gamma_{\wz{\partial}(x_0,x_{g'})}\in\Gamma_{l,l}^s$, where $(x_0,x_1,x_2,\ldots,x_{g'})$ is a path with $F(x_0)=F(x_{g'})$. Since $g'<g$, we get $x_0\neq x_{g'}$. Pick vertices $z_1,z_2\ldots,z_{s-1}$ such that $(z_i,z_{i+1})\in\Gamma_{l,l}$ for $0\leq i\leq s-1$, where $z_0=x_0$ and $z_s=x_{g'}$. The fact $p_{(1,g-1),(g-1,1)}^{(l,l)}=2$ implies that there exist vertices $y_0,y_1,y_2,y_3$ such that $P_{(1,g-1),(g-1,1)}(z_0,z_1)=\{y_0,y_1\}$ and $P_{(1,g-1),(g-1,1)}(z_{s-1},z_s)=\{y_2,y_3\}$.

Since $\Gamma_{l,l}$ is symmetric in the association scheme $(V\Gamma,R)$, one has $g'>2$. If $x_1\in\{y_0,y_1\}$, then $(z_1,x_1,x_2,\ldots,x_{g'})$ is a path with $F(z_1)=F(x_{g'})$, contrary to the minimality of $s$. Hence, $x_1\notin\{y_0,y_1\}$. Since $p_{(1,g-1),(1,g-1)}^{\wz{h}}=2$ for all $\Gamma_{\wz{h}}\in\Gamma_{1,g-1}^2$, there exists $x_1'\in P_{(1,g-1),(1,g-1)}(x_0,x_2)$ with $\Gamma_{1,g-1}(x_0)=\{y_0,y_1,x_1,x_1'\}$. Note that $p_{(1,g-1),(g-1,1)}^{(l,l)}=2$ and $p_{(l,l),(1,g-1)}^{(1,g-1)}=1$. The fact $p_{(1,g-1),(g-1,1)}^{\wz{j}}=1$ for all $\Gamma_{\wz{j}}\in\Gamma_{1,g-1}\Gamma_{g-1,1}\setminus\{\Gamma_{0,0},\Gamma_{l,l}\}$ implies $(x_1,x_1')\in\Gamma_{l,l}$.

Since $p_{(1,g-1),(1,g-1)}^{\wz{h}}=2$ for $\Gamma_{\wz{h}}\in\Gamma_{1,g-1}^2$, there exist $x_2'\in P_{(1,g-1),(1,g-1)}(x_1,x_3)$ with $x_2'\neq x_2$ and $x_2''\in P_{(1,g-1),(1,g-1)}(x_1',x_3)$ with $x_2''\neq x_2$. By $p_{(1,g-1),(g-1,1)}^{(l,l)}=2$ and $p_{(1,g-1),(l,l)}^{(1,g-1)}=p_{(1,g-1),(g-1,1)}^{\wz{j}}=1$ for all $\Gamma_{\wz{j}}\in\Gamma_{1,g-1}\Gamma_{g-1,1}\setminus\{\Gamma_{0,0},\Gamma_{l,l}\}$, if $x_2'=x_2''$, then $(x_2,x_2')\in\Gamma_{l,l}$; if $x_2'\neq x_2''$, then $\Gamma_{l,l}\cap\{(x_2,x_2'),(x_2,x_2''),(x_2',x_2'')\}\neq\emptyset$ since $x_2,x_2',x_2''\in\Gamma_{g-1,1}(x_3)$. Then there exist a path of length $2$ from $x_0$ to $w_2$ and a path of length $2$ from $x_0$ to $w_2'$ such that $w_2,w_2'\in\Gamma_{g-1,1}(x_3)$ and $(w_2,w_2')\in\Gamma_{l,l}$. By induction, there exist a path of length $i$ from $x_0$ to $w_{i}$ and a path of length $i$ from $x_0$ to $w_{i}'$ such that $w_{i},w_{i}'\in\Gamma_{g-1,1}(x_{i+1})$ and $(w_{i},w_{i}')\in\Gamma_{l,l}$ for $1\leq i\leq g'-1$.

Since $p_{(1,g-1),(1,g-1)}^{\wz{h}}=2$ for all $\Gamma_{\wz{h}}\in\Gamma_{1,g-1}^2$, from the minimality of $s$, there exist vertices $u_{i}\in P_{(1,g-1),(1,g-1)}(w_{g'-1},y_i)$ and $u_{i}'\in P_{(1,g-1),(1,g-1)}(w_{g'-1}',y_i)$ for $2\leq i\leq3$ with $z_{s-1},x_{g'}\notin\{u_2,u_3,u_2',u_3'\}$. The fact $p_{(l,l),(1,g-1)}^{(1,g-1)}=1$ and $P_{(l,l),(1,g-1)}(z_s,y_i)=\{z_{s-1}\}$ for $i\in\{2,3\}$ imply $u_2,u_3\notin\Gamma_{l,l}(z_s)$. Since $w_{g'-1}\in P_{(g-1,1),(1,g-1)}(z_s,u_2)\cap P_{(g-1,1),(1,g-1)}(z_s,u_3)$ and $p_{(g-1,1),(1,g-1)}^{\wz{j}}=1$ for all $\Gamma_{\wz{j}}\in\Gamma_{1,g-1}\Gamma_{g-1,1}\setminus\{\Gamma_{0,0},\Gamma_{l,l}\}$, we obtain $u_2\neq u_2'$ and $u_3\neq u_3'$. By $p_{(1,g-1),(g-1,1)}^{(l,l)}=2$ and $P_{(g-1,1),(1,g-1)}(y_2,y_3)=\{z_s,z_{s-1}\}$, $u_2,u_2',u_3,u_3'$ are pairwise distinct. Since $p_{(1,g-1),(l,l)}^{(1,g-1)}=1$, we have $u_{i}'\in P_{(l,l),(1,g-1)}(u_i,y_i)$ for $i\in\{2,3\}$. By $z_{s-1}\in P_{(l,l),(1,g-1)}(x_{g'},y_3)\cap P_{(l,l),(1,g-1)}(x_{g'},y_2)$, one gets $u_2,u_2',u_3,u_3'\notin\Gamma_{l,l}(x_{g'})$. It follows that $u_3\in P_{(1,g-1),(l,l)}(w_{g'-1},u_2)$ and $u_3'\in P_{(1,g-1),(l,l)}(w_{g'-1}',u_2')$.

By Lemma \ref{jb} (ii), we have $p_{(1,g-1),(l,l)}^{(1,g-1)}k=p_{(1,g-1),(g-1,1)}^{(l,l)}k_{l,l}$, and so $k_{l,l}=2$. Since $(u_2,u_2',u_3',u_3)$ is an undirected circuit of length $4$ in the graph $(V\Gamma,\Gamma_{l,l})$, from Lemma \ref{jb} (i) and (v), one gets $A_{l,l}^2=2A_{0,0}+2A_{a,a}$ and $F=\{\Gamma_{0,0},\Gamma_{l,l},\Gamma_{a,a}\}$ with $k_{a,a}=1$, where $\wz{\partial}(u_2,u_3')=(a,a)$. It follows from the minimality of $s$ that $s\in\{1,2\}$.

Since $p_{(1,g-1),(g-1,1)}^{(l,l)}=2$, there exists $z\in P_{(1,g-1),(g-1,1)}(w_{g'-1},w_{g'-1}')$ with $z\neq x_{g'}$. The fact $p_{(1,g-1),(g-1,1)}^{\wz{j}}=1$ for all $\Gamma_{\wz{j}}\in\Gamma_{1,g-1}\Gamma_{g-1,1}\setminus\{\Gamma_{0,0},\Gamma_{l,l}\}$ implies $(z,x_{g'})\in\Gamma_{l,l}$. If $s=2$, from $A_{l,l}^2=2A_{0,0}+2A_{a,a}$, then $(x_0,x_{g'})\in\Gamma_{a,a}$ and $(x_0,z)\in\Gamma_{l,l}$, which imply that there is a path of length $g'$ from $x_0$ to $z$ with $F(x_0)=F(z)$, contrary to the minimality of $s$. Thus, $s=1$ and $x_0=z_{s-1}$. Since $A_{l,l}^2=2A_{0,0}+2A_{a,a}$, we get $(x_0,z)\in\Gamma_{a,a}$. Note that $a=\partial(x_0,z)\leq g'$ and $l=\partial(x_0,x_{g'})\leq g'$. The fact that the girth of the quotient digraph of $\Gamma/F$ is $g'$ implies $l=a=g'$, contrary to the fact that $k_{l,l}=2$ and $k_{a,a}=1$.

\textbf{Case 2.2.} $g'=g$.

Since $p_{(1,g-1),(g-1,1)}^{\wz{j}}=1$ for all $\Gamma_{\wz{j}}\in\Gamma_{1,g-1}\Gamma_{g-1,1}\setminus\{\Gamma_{0,0},\Gamma_{l,l}\}$, from the proof of \cite[Proposition 4.3]{KSW03}, the quotient digraph $\Gamma/F$ is isomorphic to the digraph $\textrm{Cay}(\mathbb{Z}_g\times\mathbb{Z}_g,\{(1,0),(0,1)\})$. Let $\sigma$ be an isomorphism from the digraph $\textrm{Cay}(\mathbb{Z}_g\times\mathbb{Z}_g,\{(1,0),(0,1)\})$ to $\Gamma/F$. Pick $x_{i,j}\in \sigma(i,j)$ for all $(i,j)\in\mathbb{Z}_{g}\times\mathbb{Z}_g$. Without loss of generality, we may assume that $(x_{i,j},x_{i,j+1},\ldots,x_{i,j-1})$ and $(x_{i,j},x_{i+1,j},\ldots,x_{i-1,j})$ are two circuits for all $i,j\in\mathbb{Z}_g$. Note that $(x_{0,0},x_{2,2}),(x_{0,0},x_{3,1})\in\Gamma_{4,2g-4}$. But $x_{2,0}\in P_{(2,g-2),(2,g-2)}(x_{0,0},x_{2,2})$ and $P_{(2,g-2),(2,g-2)}(x_{0,0},x_{3,1})=\emptyset$, a contradiction.

\begin{step}\label{step2}
$p_{(2,g-2),(g-1,1)}^{(1,g-1)}=3$.
\end{step}
Let $(x_0,x_1,x_{2},\ldots,x_{g-1})$ be a circuit. Since $p_{(1,g-1),(1,g-1)}^{(2,g-2)}=2$, there exists $x_1'\in P_{(1,g-1),(1,g-1)}(x_0,x_2)$ with $x_1\neq x_1'$. Note that $x_1,x_1'\in P_{(g-1,1),(2,g-2)}(x_2,x_3)$, and so $p_{(2,g-2),(g-1,1)}^{(1,g-1)}\geq2$. Lemma \ref{jb} (iii) and Step \ref{step1} imply $p_{(2,g-2),(g-1,1)}^{(1,g-1)}=2$ or $3$.

Suppose, to the contrary that $p_{(2,g-2),(g-1,1)}^{(1,g-1)}=2$. Since $p_{(1,g-1),(1,g-1)}^{(2,g-2)}=2$, there exists $x_2'\in P_{(1,g-1),(1,g-1)}(x_1,x_3)$ with $x_2'\neq x_2$.  Since $(x_0,x_2')\in\Gamma_{2,g-2}$, there exists $x_1''\in P_{(1,g-1),(1,g-1)}(x_0,x_2')$ with $x_{1}''\neq x_1$, which implies $x_1,x_1',x_1''\in P_{(2,g-2),(g-1,1)}(x_{g-1},x_0)$. By $p_{(2,g-2),(g-1,1)}^{(1,g-1)}=2$, we get $x_1''=x_1'$. Since $x_2,x_2'\in P_{(1,g-1),(g-1,1)}(x_1,x_1')$, there exists $\Gamma_{\wz{h}}\in\Gamma_{1,g-1}\Gamma_{g-1,1}\setminus\{\Gamma_{0,0}\}$ with $p_{(1,g-1),(g-1,1)}^{\wz{h}}\geq2$.

By Lemma \ref{jb} (iii), one has $\sum_{\wz{f}\neq(2,g-2)}p_{(g-1,1),\wz{f}}^{(1,g-1)}=2$ and $\sum_{\wz{l}\neq(0,0)}p_{\wz{l},(1,g-1)}^{(1,g-1)}=3$. Setting $\wz{d}=\wz{e}=\wz{h}=\wz{g}^*=(1,g-1)$ in Lemma \ref{jb} (iv), we obtain
\begin{align}\label{g>4-1}
4+\sum_{\wz{f}\neq(2,g-2)}p_{(1,g-1),(1,g-1)}^{\wz{f}}p_{(g-1,1),\wz{f}}^{(1,g-1)}=4+\sum_{\wz{l}\neq(0,0)}p_{(g-1,1),(1,g-1)}^{\wz{l}}p_{\wz{l},(1,g-1)}^{(1,g-1)}\geq8.
\end{align}
Step \ref{step1} implies $p_{(1,g-1),(1,g-1)}^{\wz{i}}\geq3$ for some $\Gamma_{\wz{i}}\in\Gamma_{1,g-1}^2$. By Lemma \ref{jb} (ii), we get $p_{(1,g-1),(1,g-1)}^{\wz{i}}k_{\wz{i}}=4p_{\wz{i},(g-1,1)}^{(1,g-1)}$. Since $p_{(g-1,1),\wz{i}}^{(1,g-1)}\leq2$, we have $p_{(1,g-1),(1,g-1)}^{\wz{i}}=4$.

Similarly, there exists a vertex $x_0'\in P_{(1,g-1),(1,g-1)}(x_{g-1},x_1)$ with $x_0\neq x_0'$ and $(x_0',x_1')\in\Gamma_{1,g-1}$.
Since $P_{(2,g-2),(g-1,1)}(x_0,x_1)=P_{(2,g-2),(g-1,1)}(x_0',x_1)=\{x_2,x_2'\}$, there exist $x_2''\in P_{\wz{i},(g-1,1)}(x_0,x_1)$ and $x_2'''\in P_{\wz{i},(g-1,1)}(x_0',x_1)$ with $x_2'',x_2'''\notin\{x_2,x_2'\}$. If $x_2'''=x_2''$, then $x_0,x_0'\in P_{(g-1,1),\wz{i}}(x_1,x_2'')$, which implies $p_{(g-1,1),\wz{i}}^{(1,g-1)}=2$ since $p_{(g-1,1),\wz{i}}^{(1,g-1)}\leq2$. Without loss of generality, we may assume that $x_2''\neq x_2'''$. By $p_{(1,g-1),(1,g-1)}^{\wz{i}}=4$, we have $x_2'',x_2'''\in\Gamma_{1,g-1}(x_1')$.

Since $P_{(1,g-1),(g-1,1)}(x_1,x_1')=\{x_2,x_2',x_2'',x_2'''\}$, we get $p_{(1,g-1),(g-1,1)}^{\wz{j}}=4$ for some $\wz{j}\neq(0,0)$. Pick a vertex $x_3'\in P_{\wz{i},(g-1,1)}(x_1',x_2)$. Since $p_{(1,g-1),(1,g-1)}^{\wz{i}}=4$, one has $x_{2}',x_2'',x_2'''\in\Gamma_{g-1,1}(x_3')$, which implies $P_{(1,g-1),(1,g-1)}(x_1,x_3')=\{x_2,x_2',x_2'',x_2'''\}$. Since $p_{(1,g-1),(1,g-1)}^{(2,g-2)}=2$ and $\sum_{\wz{f}\neq(2,g-2)}p_{(g-1,1),\wz{f}}^{(1,g-1)}=2$, one gets $p_{(1,g-1),(1,g-1)}^{\wz{f}}=4$ for all $\Gamma_{\wz{f}}\in\Gamma_{1,g-1}^2\setminus\{\Gamma_{2,g-2}\}$.

By \eqref{g>4-1}, we have $\sum_{\wz{l}\neq(0,0)}p_{(g-1,1),(1,g-1)}^{\wz{l}}p_{\wz{l},(1,g-1)}^{(1,g-1)}=8$. Since $p_{(1,g-1),(g-1,1)}^{\wz{j}}=4$ and $\sum_{\wz{l}\neq(0,0)}p_{\wz{l},(1,g-1)}^{(1,g-1)}=3$, one gets $p_{(1,g-1),\wz{j}}^{(1,g-1)}=1$ and $p_{(1,g-1),(g-1,1)}^{\wz{h}}<4$ for $\wz{h}\notin\{(0,0),\wz{j}\}$. Lemma \ref{jb} (ii) implies $p_{(1,g-1),\wz{j}}^{(1,g-1)}k=p_{(1,g-1),(g-1,1)}^{\wz{j}}k_{\wz{j}}$, and so $k_{\wz{j}}=1$.

Let $F=\langle\Gamma_{\wz{j}}\rangle$. By $k=4$, the valency of the quotient digraph $\Gamma/F$ is $2$. Since $p_{(1,g-1),(g-1,1)}^{\wz{h}}<4$ for $\wz{h}\notin\{(0,0),\wz{j}\}$, from the proof of \cite[Proposition 4.3]{KSW03}, the quotient digraph $\Gamma/F$ is isomorphic to $\textrm{Cay}(\mathbb{Z}_g\times\mathbb{Z}_g,\{(1,0),(0,1)\})$. The fact that  $p_{(1,g-1),(g-1,1)}^{\wz{j}}=4$ and $k_{\wz{j}}=1$ imply that there exists an isomorphism $\sigma$ from $\textrm{Cay}(\mathbb{Z}_g\times\mathbb{Z}_g\times\mathbb{Z}_2,\{(1,0,0),(0,1,0),(1,0,1),(0,1,1)\})$ to $\Gamma$.
Observe that $(\sigma(0,0,0),\sigma(2,2,0)),(\sigma(0,0,0),\sigma(3,1,0))\in\Gamma_{4,2g-4}$. But one can verify $\sigma(2,0,0)\in P_{(2,g-2),(2,g-2)}(\sigma(0,0,0),\sigma(2,2,0))$ and $P_{(2,g-2),(2,g-2)}(\sigma(0,0,0),\sigma(3,1,0))=\emptyset$. This is a contradiction.

\begin{step}\label{step3}
If $p_{(1,g-1),(1,g-1)}^{(i,j)}=4$, then $p_{(1,g-1),(g-1,1)}^{\wz{h}}=2$ for $\Gamma_{\wz{h}}\in\Gamma_{1,g-1}\Gamma_{g-1,1}\setminus\{\Gamma_{0,0}\}$.
\end{step}

Suppose $p_{(1,g-1),(g-1,1)}^{\wz{h}}=4$ for some $\Gamma_{\wz{h}}\in\Gamma_{1,g-1}\Gamma_{g-1,1}\setminus\{\Gamma_{0,0}\}$. Let $(z,z')\in\Gamma_{\wz{h}}$ and $P_{(g-1,1),(1,g-1)}(z,z')=\{y,y',y'',y'''\}$. Pick a vertex $x\in P_{(g-1,1),(i,j)}(y,z)$. Since $p_{(1,g-1),(1,g-1)}^{(i,j)}=4$, we get $y',y'',y'''\in\Gamma_{1,g-1}(x)$. The fact $p_{(1,g-1),(1,g-1)}^{(2,g-2)}=2$ implies $z,z'\notin P_{(2,g-2),(g-1,1)}(x,y)$, and so $p_{(1,g-1),(g-2,2)}^{(1,g-1)}\leq2$, contrary to Step \ref{step2}. Hence, $p_{(1,g-1),(g-1,1)}^{\wz{h}}\leq3$ for all $\Gamma_{\wz{h}}\in\Gamma_{1,g-1}\Gamma_{g-1,1}\setminus\{\Gamma_{0,0}\}$.

Since $p_{(g-1,1),(2,g-2)}^{(1,g-1)}=3$ from Step \ref{step2}, by Lemma \ref{jb} (iii), we have $p_{(g-1,1),(i,j)}^{(1,g-1)}=1$. By setting $\wz{d}=\wz{e}=\wz{h}=\wz{g}^*=(1,g-1)$ in Lemma \ref{jb} (iv), we have
\begin{align}\label{eq-step3}
10=\sum_{\wz{f}}p_{(1,g-1),(1,g-1)}^{\wz{f}}p_{(g-1,1),\wz{f}}^{(1,g-1)}=4+\sum_{\wz{l}\neq(0,0)}p_{(g-1,1),(1,g-1)}^{\wz{l}}p_{\wz{l},(1,g-1)}^{(1,g-1)}.
\end{align}

If $p_{(1,g-1),(g-1,1)}^{\wz{h}}=3$ for some $\Gamma_{\wz{h}}\in\Gamma_{1,g-1}\Gamma_{g-1,1}$, from \eqref{eq-step3}, then $p_{\wz{h},(1,g-1)}^{(1,g-1)}=1$, which implies $p_{(1,g-1),(g-1,1)}^{\wz{h}}k_{\wz{h}}=p_{\wz{h},(1,g-1)}^{(1,g-1)}k=4$ by Lemma \ref{jb} (ii), a contradiction. Since $\sum_{\wz{l}\neq(0,0)}p_{\wz{l},(1,g-1)}^{(1,g-1)}=3$ from Lemma \ref{jb} (iii), the desired result follows by \eqref{eq-step3}.

\begin{step}\label{step4}
$g=4$ and $k_{2,2}=6$.
\end{step}

Since $p_{(2,g-2),(g-1,1)}^{(1,g-1)}=3$ from Step \ref{step2}, by Lemma \ref{jb} (ii), we get $p_{(2,g-2),(g-1,1)}^{(1,g-1)}k=p_{(1,g-1),(1,g-1)}^{(2,g-2)}k_{2,g-2}$, and so $k_{2,g-2}=6$. It suffices to show that $g=4$.

Since $p_{(2,g-2),(g-1,1)}^{(1,g-1)}=3$, from Lemma \ref{jb} (iii), we have $p_{(i,j),(g-1,1)}^{(1,g-1)}=1$ for some $(i,j)\neq(2,g-2)$. By setting $\wz{d}=\wz{e}=(1,g-1)$ in Lemma \ref{jb} (i), one gets $p_{(1,g-1),(1,g-1)}^{(i,j)}k_{i,j}=4$. Step \ref{step1} implies $p_{(1,g-1),(1,g-1)}^{(i,j)}=1$ or $4$.

Let $(x,z)\in\Gamma_{2,g-2}$ and $P_{(1,g-1),(1,g-1)}(x,z)=\{y,y'\}$. We claim $(y,z')\notin\Gamma_{1,g-1}$ for $z'\in P_{(2,g-2),(g-1,1)}(x,y')\setminus\{z\}$ and $(z'',y)\notin\Gamma_{1,g-1}$ for $z''\in P_{(g-1,1),(2,g-2)}(y',z)\setminus\{x\}$. Suppose $(y,z')\in\Gamma_{1,g-1}$ for some $z'\in P_{(2,g-2),(g-1,1)}(x,y')\setminus\{z\}$. Since $z,z'\in P_{(1,g-1),(g-1,1)}(y,y')$, we have $p_{(1,g-1),(g-1,1)}^{\wz{l}}\geq2$ for some $\wz{l}\neq(0,0)$. In view of Lemma \ref{jb} (iii), we have $\sum_{\wz{l}\neq(0,0)}p_{\wz{l},(1,g-1)}^{(1,g-1)}=3$. By setting $\wz{d}=\wz{e}=\wz{h}=\wz{g}^*=(1,g-1)$ in Lemma \ref{jb} (iv), one gets $6+p_{(1,g-1),(1,g-1)}^{(i,j)}=\sum_{\wz{l}}p_{(g-1,1),(1,g-1)}^{\wz{l}}p_{\wz{l},(1,g-1)}^{(1,g-1)}\geq8$, which implies $p_{(1,g-1),(1,g-1)}^{(i,j)}=4$. The fact $p_{(i,j),(g-1,1)}^{(1,g-1)}=1$ implies that there exists a vertex $z'''\in P_{(i,j),(g-1,1)}(x,y')$. Since $p_{(1,g-1),(1,g-1)}^{(i,j)}=4$, we obtain $y\in P_{(1,g-1),(1,g-1)}(x,z''')$, and so $z,z',z'''\in P_{(1,g-1),(g-1,1)}(y,y')$, contrary to Step \ref{step3}. Hence, $(y,z')\notin\Gamma_{1,g-1}$ for $z'\in P_{(2,g-2),(g-1,1)}(x,y')\setminus\{z\}$. Similarly, $(z'',y)\notin\Gamma_{1,g-1}$ for $z''\in P_{(g-1,1),(2,g-2)}(y',z)\setminus\{x\}$. Thus, our claim is valid.

Let $(x_0=x_g,x_1,\ldots,x_{g-1})$ be a circuit. Since $p_{(1,g-1),(1,g-1)}^{(2,g-2)}=2$, there exist vertices $x_2'\in P_{(1,g-1),(1,g-1)}(x_1,x_3)$ with $x_2'\neq x_2$, $x_1'\in P_{(1,g-1),(1,g-1)}(x_0,x_2')$ with $x_1'\neq x_1$ and $x_3'\in P_{(1,g-1),(1,g-1)}(x_2',x_4)$ with $x_3'\neq x_3$. By the claim, we have $(x_1',x_2)\notin\Gamma_{1,g-1}$. Since $p_{(1,g-1),(1,g-1)}^{(2,g-2)}=2$, there exists $x_2''\in P_{(1,g-1),(1,g-1)}(x_1',x_3)$ with $x_2''\notin\{x_2,x_2'\}$. By the claim, we get $(x_2'',x_3')\notin\Gamma_{1,g-1}$. Since $p_{(1,g-1),(1,g-1)}^{(2,g-2)}=2$, there exists $x_3''\in P_{(1,g-1),(1,g-1)}(x_2'',x_4)$ with $x_3''\notin\{x_3,x_3'\}$. The claim implies $x_3',x_3''\notin\Gamma_{1,g-1}(x_2)$. It follows that there exists $x_3'''\in P_{(1,g-1),(1,g-1)}(x_2,x_4)$ with $\Gamma_{g-1,1}(x_4)=\{x_3,x_3',x_3'',x_3'''\}$. If $g>4$, then $\Gamma_{g-1,1}(x_4)\subseteq P_{(g-1,1),(2,g-2)}(x_4,x_5)$, contrary to Step \ref{step2}. Hence, $g=4$.

\begin{step}\label{step5}
$\Gamma_{1,3}^2=\{\Gamma_{2,2},\Gamma_{i,j}\}$, $p_{(i,j),(3,1)}^{(1,3)}=p_{(1,3),(1,3)}^{(i,j)}=1$ and $k_{i,j}=4$ with $(i,j)\neq(2,2)$.
\end{step}

By Step \ref{step4}, we have $k_{2,2}=6$. In view of Lemma \ref{jb} (iii) and Step \ref{step2}, we have $\Gamma_{1,3}^2=\{\Gamma_{2,2},\Gamma_{i,j}\}$ and $p_{(i,j),(3,1)}^{(1,3)}=1$ for some $(i,j)\neq(2,2)$. By setting $\wz{d}=\wz{e}=(1,3)$ in Lemma \ref{jb} (i), one gets $p_{(1,3),(1,3)}^{(i,j)}k_{i,j}=4$. Step \ref{step1} implies $p_{(1,3),(1,3)}^{(i,j)}=1$ or $4$.

Assume the contrary, namely, $p_{(1,3),(1,3)}^{(i,j)}=4$. Note that $k_{i,j}=1$. Since $k=4$, from Lemma \ref{jb} (iii), one gets $i=2$. Let $\Gamma_{a,b}\in\Gamma_{1,3}\Gamma_{3,1}$ with $(a,b)\neq(0,0)$. Since $\Gamma_{1,3}\notin\Gamma_{1,3}^2$, from Lemma \ref{jb2}, we have $1<a,b\leq4$.

Suppose $a<4$ or $b<4$. Without loss of generality, we may assume $a<4$. Let $(x,w)\in\Gamma_{a,b}$. By Step \ref{step3}, there exist vertices $u,u'$ such that $P_{(3,1),(1,3)}(x,w)=\{u,u'\}$. Note that $p_{(1,3),(1,3)}^{(2,j)}=4$. By Lemma \ref{jb} (iii), one gets $\Gamma_{2,j}\notin\Gamma_{1,3}\Gamma_{3,1}$. Since $\Gamma_{1,3}^2=\{\Gamma_{2,2},\Gamma_{2,j}\}$, if $a=2$, then $(a,b)=(2,2)$ and there exist distinct vertices $y,y'\in P_{(1,3),(1,3)}(x,w)$ and $z\in P_{(1,3),(1,3)}(w,x)$, which imply $y,y',u,u'\in P_{(3,1),(2,2)}(w,z)$ since $x\in P_{(1,3),(3,1)}(z,u)\cap P_{(1,3),(3,1)}(z,u')$, contrary to Step \ref{step2}. Hence, $a=3$.

Let $(x,y,z,w)$ be a path. If $(y,w)\in\Gamma_{2,j}$, from $p_{(1,3),(1,3)}^{(2,j)}=4$, then $u\in P_{(1,3),(1,3)}(y,w)$, which implies that $(u,x,y)$ is a circuit, a contradiction. Since $\Gamma_{1,3}^2=\{\Gamma_{2,2},\Gamma_{2,j}\}$, from the commutativity of $\Gamma$, we have $(x,z),(y,w)\in\Gamma_{2,2}$. Similarly, $(u,y)\in\Gamma_{2,2}$. By $p_{(1,3),(1,3)}^{(2,2)}=2$, there exists $x'\in P_{(1,3),(1,3)}(u,y)\setminus\{x\}$ with $(x',z)\in\Gamma_{2,2}$. Since $P_{(1,3),(1,3)}(w,y)\subseteq P_{(3,1),(2,2)}(y,z)$ and $p_{(3,1),(2,2)}^{(1,3)}=3$, we get $x$ or $x'\in P_{(1,3),(1,3)}(w,y)$, contrary to the fact that $\Gamma_{1,3}\notin\Gamma_{1,3}^2$. Thus, $a=b=4$.

Note that $\Gamma_{1,3}\Gamma_{3,1}=\{\Gamma_{0,0},\Gamma_{4,4}\}$. Since $p_{(1,3),(3,1)}^{(4,4)}=2$ from Step \ref{step3}, by setting $\wz{d}=\wz{e}^*=(1,3)$ in Lemma \ref{jb} (i), we have $k_{4,4}=6$. Since $k_{2,j}=1$, from Lemma \ref{jb} (i) and (vi), one gets $\Gamma_{2,j}\Gamma_{1,3}=\{\Gamma_{i',j-1}\}$ with $k_{i',j-1}=4$. By Lemma \ref{jb} (ii), one gets $p_{(2,j),(1,3)}^{(1,3)}=p_{(1,3),(3,1)}^{(2,j)}=0$, which implies $i'=3$ from $k_{2,2}=6$. Let $(x,z)\in\Gamma_{2,2}$ and $\Gamma_{1,3}(z)=\{w,w',w'',w'''\}$. Since $p_{(1,2),(1,2)}^{(2,2)}=2$, we may assume $P_{(1,3),(1,3)}(x,z)=\{y,y'\}$ and $P_{(1,3),(1,3)}(z,x)=\{w'',w'''\}$. The fact $w'',w'''\in P_{(2,2),(3,1)}(y,z)\cap P_{(2,2),(3,1)}(y',z)$ and $p_{(2,j),(3,1)}^{(1,3)}=1$ imply $(y,w),(y',w')\in\Gamma_{2,j}$ or $(y,w'),(y',w)\in\Gamma_{2,j}$. Since $\Gamma_{2,j}\Gamma_{1,3}=\{\Gamma_{3,j-1}\}$, we have $w,w'\in\Gamma_{3,j-1}(x)$. Then $\Gamma_{2,2}\Gamma_{1,3}=\{\Gamma_{3,1},\Gamma_{3,j-1}\}$, and so $\Gamma_{1,3}^3=\{\Gamma_{3,1},\Gamma_{3,j-1}\}$.

Pick a path $(x_0,x_1,x_2,x_3,x_4)$ with $(x_0,x_2),(x_2,x_4)\in\Gamma_{2,j}$. Since $k_{2,j}=1$, from Lemma \ref{jb} (i), we have $k_{c,d}=1$ with $\wz{\partial}(x_0,x_4)=(c,d)$. The fact $\Gamma_{1,3}\Gamma_{3,1}=\{\Gamma_{0,0},\Gamma_{4,4}\}$ implies $x_1\notin P_{(1,3),(1,3)}(x_0,x_4)$. Since $p_{(1,3),(1,3)}^{(2,j)}=4$, one gets $(c,d)\neq(2,j)$. Note that $\Gamma_{1,3}^2=\{\Gamma_{2,2},\Gamma_{2,j}\}$ and $\Gamma_{1,3}^3=\{\Gamma_{3,1},\Gamma_{3,j-1}\}$. Since $k_{2,2}=6$ and $k_{1,3}=k_{3,j-1}=4$, one gets $c=4$.

Since $p_{(1,3),(1,3)}^{(2,j)}=4$, we may assume $(x_1,x_3)\in\Gamma_{2,2}$. The fact $p_{(1,3),(1,3)}^{(2,2)}=2$ implies that there exists $x_2'\in P_{(1,3),(1,3)}(x_1,x_3)$ with $x_2'\neq x_2$. Since $p_{(2,j),(3,1)}^{(1,3)}=1$, one has $(x_0,x_2'),(x_2',x_4)\in\Gamma_{2,2}$, which implies that there exist $x_1'\in P_{(1,3),(1,3)}(x_2',x_0)$ and $x_3'\in P_{(1,3),(1,3)}(x_4,x_2')$. Since $k_{4,4}=6$ and $k_{4,d}=1$, we get $d<4$. Note that $\Gamma_{1,3}^2=\{\Gamma_{2,2},\Gamma_{2,j}\}$ and $\Gamma_{1,3}^3=\{\Gamma_{3,1},\Gamma_{3,j-1}\}$. Since $k_{2,2}=6$ and $k_{1,3}=k_{3,j-1}=4$, one obtains $d=2$ and $j=c=4$. By $p_{(1,3),(1,3)}^{(2,j)}=4$, we have $x_1'\in P_{(1,3),(1,3)}(x_4,x_0)$, which implies $x_1'\in P_{(1,3),(3,1)}(x_2',x_4)$, contrary to the fact that $\Gamma_{2,2}\notin\Gamma_{1,3}\Gamma_{3,1}$.

This completes the proof of this step.

\begin{step}\label{step6}
For each $\Gamma_{\wz{l}}\in\Gamma_{1,3}\Gamma_{3,1}$ with $\wz{l}\neq(0,0)$, $p_{(1,3),(3,1)}^{\wz{l}}=1$.
\end{step}
In view of Step \ref{step5}, we have $p_{(1,3),(1,3)}^{(i,j)}=p_{(i,j),(3,1)}^{(1,3)}=1$.
By Step \ref{step2} and setting $\wz{d}=\wz{e}=\wz{h}=\wz{g}^*=(1,3)$ in Lemma \ref{jb} (iv), we have $7=\sum_{\wz{f}}p_{(1,3),(1,3)}^{\wz{f}}p_{(3,1),\wz{f}}^{(1,3)}=4+\sum_{\wz{l}\neq(0,0)}p_{(3,1),(1,3)}^{\wz{l}}p_{\wz{l},(1,3)}^{(1,3)}$. In view of Lemma \ref{jb} (iii), we get $\sum_{\wz{l}\neq(0,0)}p_{\wz{l},(1,3)}^{(1,3)}=3$, and so $p_{(1,3),(3,1)}^{\wz{l}}=1$ for all $\Gamma_{\wz{l}}\in\Gamma_{1,3}\Gamma_{3,1}$ with $\wz{l}\neq(0,0)$.

\begin{step}\label{step7}
$i=2$.
\end{step}

Suppose $i=1$. Note that $j=3$ and $\Gamma_{1,3}\in\Gamma_{1,3}^2$. Let $(x,z)\in\Gamma_{1,3}$ and $y\in P_{(1,3),(1,3)}(x,z)$. Pick a vertex $w\in P_{(1,3),(3,1)}(x,z)$. Since $p_{(1,3),(3,1)}^{(1,3)}=1$ from Step \ref{step6}, we have $(y,w)\in\Gamma_{2,2}$ by Step \ref{step5}. The fact $x\in P_{(3,1),(1,3)}(y,w)$ implies $\Gamma_{2,2}\in\Gamma_{1,3}\Gamma_{3,1}$. Note that $p_{(1,3),(3,1)}^{(2,2)}k_{2,2}=6$, contrary to Lemma \ref{jb} (v). Thus, $i=2$.

\begin{step}\label{jb g=4}
Suppose $(x,z)\in\Gamma_{2,2}$ and $P_{(1,3),(1,3)}(x,z)=\{y,y'\}$. Then $P_{(2,2),(3,1)}(y,z)\cap P_{(2,2),(3,1)}(y',z)\subseteq\Gamma_{3,1}(x)$ and $P_{(3,1),(2,2)}(x,y)\cap P_{(3,1),(2,2)}(x,y')\subseteq\Gamma_{1,3}(z)$.
\end{step}

Let $w\in P_{(2,2),(3,1)}(y,z)\cap P_{(2,2),(3,1)}(y',z)$. Since $p_{(2,j),(3,1)}^{(1,3)}=1$ from Step \ref{step5}, there exist distinct vertices $w'\in P_{(2,j),(3,1)}(y,z)$ and $w''\in P_{(2,j),(3,1)}(y',z)$. By $j>2$, we have $w',w''\notin P_{(1,3),(1,3)}(z,x)$. Since $p_{(1,3),(1,3)}^{(2,2)}=2$, one gets $w\in P_{(1,3),(1,3)}(z,x)$. Hence, $P_{(2,2),(3,1)}(y,z)\cap P_{(2,2),(3,1)}(y',z)\subseteq\Gamma_{3,1}(x)$. Similarly, $P_{(3,1),(2,2)}(x,y)\cap P_{(3,1),(2,2)}(x,y')\subseteq\Gamma_{1,3}(z)$.

\begin{step}\label{chain-1}
Let $(x_0,x_1,x_2,x_3)$ be a path such that $(x_0,x_2),(x_1,x_3)\in\Gamma_{2,j}$.  Suppose $\wz{\partial}(x_0,x_3)=\wz{a}$. Then $\Gamma_{\wz{a}}\notin\Gamma_{2,2}\Gamma_{1,3}$, $k_{\wz{a}}\in\{1,2,4\}$ and
\begin{align}
\Gamma_{\wz{a}}=\{(y_0,y_3)\mid~there~is~a~path~(y_0,y_1,y_2,y_3)~with~(y_0,y_2),(y_1,y_3)\in\Gamma_{2,j}\}.\nonumber
\end{align}
\end{step}

Suppose $\Gamma_{\wz{a}}\in\Gamma_{2,2}\Gamma_{1,3}$. Note that there exists $x_{2}'\in P_{(2,2),(1,3)}(x_0,x_3)$ with $x_2\neq x_{2}'$. By Step \ref{step5}, we have $p_{(1,3),(1,3)}^{(2,j)}=p_{(2,j),(3,1)}^{(1,3)}=1$. Since $p_{(1,3),(1,3)}^{(2,2)}=2$, there exists $x_1',x_1''\in P_{(1,3),(1,3)}(x_0,x_2')$ such that $x_1,x_1',x_1''$ are pairwise distinct. Since $(x_0,x_3)\notin\Gamma_{3,1}$ and $p_{(2,j),(3,1)}^{(1,3)}=1$, from Step \ref{jb g=4}, we may assume $(x_1',x_3)\in\Gamma_{2,2}$ and $(x_1'',x_3)\in\Gamma_{2,j}$. It follows that $p_{(1,2),(2,2)}^{\wz{a}}\geq1$ and $p_{(1,2),(2,j)}^{\wz{a}}\geq2$. The fact $p_{(1,3),(1,3)}^{(2,2)}=2$ and $p_{(1,3),(1,3)}^{(2,j)}=1$ imply that there exists $x_2''\in P_{(1,3),(1,3)}(x_1',x_3)$ with $x_2''\notin\{x_2,x_2'\}$. Since $(x_0,x_3)\notin\Gamma_{3,1}$, from Step \ref{jb g=4}, one gets $(x_0,x_2'')\in\Gamma_{2,j}$.

If $p_{(1,2),(2,2)}^{\wz{a}}=2$, then there exists $x_{2}'''\in P_{(2,2),(1,2)}(x_0,x_3)$ with $x_2'''\notin\{x_2,x_2',x_2''\}$, which implies $x_1$ or $x_1''\in P_{(1,3),(1,3)}(x_0,x_2''')$ since $p_{(1,3),(1,3)}^{(2,2)}=2$, contrary to the fact that $p_{(1,3),(1,3)}^{(2,j)}=1$. Lemma \ref{jb} (iii) implies $p_{(1,2),(2,2)}^{\wz{a}}=1$. By Lemma \ref{jb} (ii) and Step \ref{step4}, we have $k_{\wz{a}}=p_{(1,2),(2,2)}^{\wz{a}}k_{\wz{a}}=k_{2,2}p_{\wz{a},(3,1)}^{(2,2)}\geq6$. Since $p_{(2,j),(3,1)}^{(1,3)}=1$, there exists $(y_0,y_3)\in\Gamma_{\wz{a}}$ such that there is no path $(y_0,y_1,y_2,y_3)$ with $(y_0,y_2),(y_1,y_3)\in\Gamma_{2,j}$.

Since $p_{(1,2),(2,2)}^{\wz{a}}=1$ and $p_{(1,2),(1,2)}^{(2,2)}=2$, we may assume $y_2\in P_{(2,2),(1,3)}(y_0,y_3)$ and $P_{(1,3),(1,3)}(y_0,y_2)=\{y_1,y_1'\}$. The fact $p_{(1,2),(2,j)}^{\wz{a}}\geq2$ and $p_{(2,j),(3,1)}^{(1,3)}=1$ imply that there exist $y_1''\in P_{(1,3),(2,j)}(y_0,y_3)$ with $y_1''\notin\{y_1,y_1'\}$ and $y_2''\in P_{(1,3),(1,3)}(y_1'',y_3)$ with $y_2\neq y_2''$. Note that $(y_0,y_2'')\in\Gamma_{2,2}$. This implies $y_2,y_2''\in P_{(2,2),(1,3)}(y_0,y_3)$, contrary to the fact that $p_{(1,2),(2,2)}^{\wz{a}}=1$. Thus, $\Gamma_{\wz{a}}\notin\Gamma_{2,2}\Gamma_{1,3}$.


For any $(z_0,z_3)\in\Gamma_{\wz{a}}$, there exists a path $(z_0,z_1,z_2,z_3)$ with $(z_1,z_3)\in\Gamma_{2,j}$ by $x_1\in P_{(1,3),(2,j)}(x_0,x_3)$. Since $\Gamma_{\wz{a}}\notin\Gamma_{2,2}\Gamma_{1,3}$, we have $(z_0,z_2)\notin\Gamma_{2,2}$, and so $(z_0,z_2)\in\Gamma_{2,j}$ from Step \ref{step5}. Conversely, for any path $(z_0,z_1,z_2,z_3)$ with $(z_0,z_2),(z_1,z_3)\in\Gamma_{2,j}$, since $p_{(2,2),(3,1)}^{(1,3)}=3$ and $\Gamma_{\wz{a}}\notin\Gamma_{2,2}\Gamma_{1,3}$, we get $p_{\wz{a},(3,1)}^{(2,j)}=1$ and $(z_0,z_3)\in\Gamma_{\wz{a}}$, which imply $\Gamma_{\wz{a}}=\{(y_0,y_3)\mid {\rm there~is~a~path}~(y_0,y_1,y_2,y_3)~{\rm with}~(y_0,y_2),(y_1,y_3)\in\Gamma_{2,j}\}$.

Since $p_{\wz{a},(3,1)}^{(2,j)}=1$, from Lemma \ref{jb} (ii), we have $k_{\wz{a}}p_{(2,j),(1,3)}^{\wz{a}}=k_{2,j}p_{\wz{a},(3,1)}^{(2,j)}=4$. It follows that $k_{\wz{a}}\in\{1,2,4\}$.

This completes the proof of this step.

\begin{step}\label{chain-2}
Let $(x_0,x_1,x_2,y_3)$ be a path with $(x_0,x_2)\in\Gamma_{2,j}$ and $(x_1,y_3)\in\Gamma_{2,2}$. Suppose $\wz{\partial}(x_0,y_3)=\wz{b}$. Then $\Gamma_{2,j}\Gamma_{1,3}=\{\Gamma_{\wz{a}},\Gamma_{\wz{b}}\}$, $\Gamma_{2,2}\Gamma_{1,3}=\{\Gamma_{3,1},\Gamma_{\wz{b}}\}$, $p_{(2,2),(1,3)}^{\wz{b}}=p_{(2,j),(1,3)}^{\wz{b}}=1$ and $k_{\wz{b}}=12$.
\end{step}
Since $p_{(1,3),(1,3)}^{(2,2)}=2$, there exists $x_2'\in P_{(1,3),(1,3)}(x_1,y_3)$ with $x_2\neq x_2'$. By Step \ref{step5}, one has $p_{(2,j),(3,1)}^{(1,3)}=1$, which implies $(x_0,x_2')\in\Gamma_{2,2}$. Then there exists $x_1'\in P_{(1,3),(1,3)}(x_0,x_2')$ with $x_1\neq x_1'$. Since $(x_0,y_3)\notin\Gamma_{3,1}$, from Step \ref{jb g=4}, one has $(x_1',y_3)\in\Gamma_{2,j}$. Since $x_2'\in P_{(2,2),(1,3)}(x_0,y_3)$ and $x_2\in P_{(2,j),(1,3)}(x_0,y_3)$, we get $p_{(2,2),(1,3)}^{\wz{b}}\geq1$ and $p_{(2,j),(1,3)}^{\wz{b}}\geq1$.

\textbf{Case 1.} $p_{(2,2),(1,3)}^{\wz{b}}\geq2$ or $p_{(2,j),(1,3)}^{\wz{b}}\geq2$.

Suppose $p_{(2,j),(1,3)}^{\wz{b}}\geq2$. Then there exist $x_2''\in P_{(2,j),(1,3)}(x_0,y_3)$ with $x_2''\notin\{x_2,x_2'\}$ and $x_1''\in P_{(1,3),(1,3)}(x_0,x_2'')$ with $x_1''\notin\{x_1,x_1'\}$ since $p_{(1,3),(1,3)}^{(2,2)}=2$ and $p_{(1,3),(1,3)}^{(2,j)}=1$. By Step \ref{chain-1}, one has $(x_1'',y_3)\in\Gamma_{2,2}$, which implies that there exists $x_2'''\in P_{(1,3),(1,3)}(x_1'',y_3)$ with $x_2'''\neq x_2''$. Since $p_{(1,3),(1,3)}^{(2,j)}=1$ and $p_{(1,3),(1,3)}^{(2,2)}=2$, we get $x_2'''\notin\{x_2,x_2'\}$. The fact $p_{(2,j),(3,1)}^{(1,3)}=1$ implies $(x_0,x_2''')\in\Gamma_{2,2}$. Suppose $p_{(2,2),(1,3)}^{\wz{b}}\geq2$. Then there exist $x_1''\in P_{(1,3),(2,2)}(x_0,y_3)$ with $x_1''\notin\{x_1,x_1'\}$ and $x_2'',x_2'''\in P_{(1,3),(1,3)}(x_1'',y_3)$ with $\Gamma_{3,1}(y_3)=\{x_2,x_2',x_2'',x_2'''\}$ since $p_{(1,3),(1,3)}^{(2,2)}=2$ and $p_{(1,3),(1,3)}^{(2,j)}=1$. Since $(x_0,y_3)\notin\Gamma_{3,1}$, from Step \ref{jb g=4}, we may assume $(x_0,x_2'')\in\Gamma_{2,j}$ and $(x_0,x_2''')\in\Gamma_{2,2}$. For both cases, since $x_2',x_2'''\in P_{(2,2),(1,3)}(x_0,y_3)$ and $x_2,x_2''\in P_{(2,j),(1,3)}(x_0,y_3)$, from Lemma \ref{jb} (iii), one has $p_{(2,2),(1,3)}^{\wz{b}}=p_{(2,j),(1,3)}^{\wz{b}}=2$.

By Lemma \ref{jb} (ii) and Steps \ref{step4}, \ref{step5}, we get $2k_{\wz{b}}=k_{2,2}p_{\wz{b},(3,1)}^{(2,2)}=6p_{\wz{b},(3,1)}^{(2,2)}$ and $2k_{\wz{b}}=k_{2,j}p_{\wz{b},(3,1)}^{(2,j)}=4p_{\wz{b},(3,1)}^{(2,j)}$. Then $p_{\wz{b},(3,1)}^{(2,2)}=2$, $p_{\wz{b},(3,1)}^{(2,j)}=3$ and $k_{\wz{b}}=6$. Since $p_{(3,1),(3,1)}^{(2,2)}=2$, from Lemma \ref{jb} (iii) and (iv), one has $p_{(1,3),(2,2)}^{(3,1)}p_{(3,1),(3,1)}^{(2,2)}+p_{(1,3),(2,2)}^{\wz{b}}p_{(3,1),\wz{b}}^{(2,2)}=4+\sum_{\wz{l}\neq(0,0)}p_{(3,1),(1,3)}^{\wz{l}}p_{\wz{l},(2,2)}^{(2,2)}=10.$
Step \ref{step6} implies $\sum_{\wz{l}\neq(0,0)}p_{\wz{l},(2,2)}^{(2,2)}\geq6$, a contradiction.

\textbf{Case 2.} $p_{(2,2),(1,3)}^{\wz{b}}=p_{(2,j),(1,3)}^{\wz{b}}=1$.

By Lemma \ref{jb} (ii) and Steps \ref{step4}, \ref{step5}, we have $k_{\wz{b}}=k_{2,2}p_{\wz{b},(3,1)}^{(2,2)}=6p_{\wz{b},(3,1)}^{(2,2)}$ and $k_{\wz{b}}=k_{2,j}p_{\wz{b},(3,1)}^{(2,j)}=4p_{\wz{b},(3,1)}^{(2,j)}$, which imply that $p_{\wz{b},(3,1)}^{(2,j)}=3$. Then $p_{\wz{b},(3,1)}^{(2,2)}=2$ and $k_{\wz{b}}=12$. Since $p_{\wz{a},(3,1)}^{(2,j)}\geq1$ and $p_{(3,1),(3,1)}^{(2,2)}=2$, from Lemma \ref{jb} (iii), one gets $\Gamma_{2,j}\Gamma_{1,3}=\{\Gamma_{\wz{a}},\Gamma_{\wz{b}}\}$ and $\Gamma_{2,2}\Gamma_{1,3}=\{\Gamma_{3,1},\Gamma_{\wz{b}}\}$.

\begin{step}\label{chain-3}
Let $(z_0,z_1,z_2,z_3,z_4)$ be a path with $(z_0,z_2),(z_2,z_4)\in\Gamma_{2,j}$ and $(z_1,z_3)\in\Gamma_{2,2}$. Then $(z_0,z_4)\in\Gamma_{4,4}$ and $k_{4,4}\in\{3,6\}$.
\end{step}
Since $p_{(1,3),(1,3)}^{(2,2)}=2$, there exists $z_2'\in P_{(1,3),(1,3)}(z_1,z_3)$ with $z_2\neq z_2'$. By Step \ref{step5}, we have $p_{(2,j),(3,1)}^{(1,3)}=p_{(1,3),(1,3)}^{(2,j)}=1$ and $k_{2,j}=4$, which imply $z_0,z_4\in\Gamma_{2,2}(z_2')$.

We claim that there exists a path $(z_4,w_3,w_2,w_1,z_0)$ such that $(z_4,w_2),(w_2,z_0)\in\Gamma_{2,j}$ and $(w_3,w_1)\in\Gamma_{2,2}$ if $k_{\wz{c}}\in\{3,6\}$ with $\wz{\partial}(z_0,z_4)=\wz{c}$. Let $P_{(1,3),(1,3)}(z_4,z_2')=\{w_3,w_3'\}$ and $P_{(1,3),(1,3)}(z_2',z_0)=\{w_1,w_1'\}$. Since $k_{\wz{c}}\in\{3,6\}$ and $p_{(1,3),(3,1)}^{\wz{l}}\leq1$ for all $\wz{l}\neq(0,0)$ from Step \ref{step6}, by Lemma \ref{jb} (v), we have $\Gamma_{\wz{c}}\notin\Gamma_{1,3}\Gamma_{3,1}$, which implies $w_1,w_1'\notin\Gamma_{3,1}(z_4)$. Since $p_{(2,j),(3,1)}^{(1,3)}=1$, from Step \ref{jb g=4}, we may assume $(w_3',w_1),(w_3,w_1')\in\Gamma_{2,j}$ and $(w_3,w_1)\in\Gamma_{2,2}$. The fact $p_{(1,3),(1,3)}^{(2,2)}=2$ implies that there exists $w_2\in P_{(1,3),(1,3)}(w_3,w_1)$ with $w_2\neq z_2'$. Since $(w_3,z_0),(z_4,w_1)\notin\Gamma_{3,1}$, from Step \ref{jb g=4}, we obtain $w_2\in P_{(2,j),(2,j)}(z_4,z_0)$. Thus, our claim is valid.

By the claim, it suffices to show $\partial(z_0,z_4)=4$ and $k_{\wz{c}}\in\{3,6\}$. Since $p_{(1,3),(1,3)}^{(2,2)}=2$, there exist $z_1'\in P_{(1,3),(1,3)}(z_0,z_2')$ with $z_1'\neq z_1$ and $z_3'\in P_{(1,3),(1,3)}(z_2',z_4)$ with $z_3'\neq z_3$. Since $(z_0,z_2),(z_2,z_4)\in\Gamma_{2,j}$, one has $(z_0,z_3),(z_1,z_4)\notin\Gamma_{3,1}$, which implies $(z_1',z_3),(z_1,z_3')\in\Gamma_{2,j}$ from Step \ref{jb g=4}. By $p_{(2,j),(3,1)}^{(1,3)}=1$, one has $(z_1',z_3')\in\Gamma_{2,2}$. Then there exists $z_2''\in P_{(1,3),(1,3)}(z_1',z_3')$ with $z_2''\notin\{z_2,z_2'\}$. Since $(z_0,z_3'),(z_1',z_4)\notin\Gamma_{3,1}$, from Step \ref{jb g=4}, we get $(z_0,z_2''),(z_2'',z_4)\in\Gamma_{2,j}$. The fact $z_2,z_2''\in P_{(2,j),(2,j)}(z_0,z_4)$ and $z_2'\in P_{(2,2),(2,2)}(z_0,z_4)$ imply $p_{(2,j),(2,j)}^{\wz{c}}\geq2$ and $p_{(2,2),(2,2)}^{\wz{c}}\geq1$.

By Lemma \ref{jb} (ii) and Step \ref{step4}, we have
\begin{align}
k_{\wz{c}}p_{(2,j),(2,j)}^{\wz{c}}=k_{2,j}p_{\wz{c},(j,2)}^{(2,j)}=4p_{\wz{c},(j,2)}^{(2,j)},\label{eq-1}\\
k_{\wz{c}}p_{(2,2),(2,2)}^{\wz{c}}=k_{2,2}p_{\wz{c},(2,2)}^{(2,2)}=6p_{\wz{c},(2,2)}^{(2,2)}.\label{eq-2}
\end{align}

{\bf Case 1.} $p_{(2,2),(2,2)}^{\wz{c}}=1$.

By \eqref{eq-2}, we have $k_{\wz{c}}=6p_{\wz{c},(2,2)}^{(2,2)}$. In view of \eqref{eq-1}, one gets $p_{\wz{c},(j,2)}^{(2,j)}=3$, which implies $k_{\wz{c}}=6$ since $p_{(2,j),(2,j)}^{\wz{c}}\geq2$.

Assume the contrary, namely, $\partial(z_0,z_4)<4$. By Step \ref{step5}, we get
$\Gamma_{1,3}^2=\{\Gamma_{2,2},\Gamma_{2,j}\}$ with $k_{2,j}=4$.  Steps \ref{chain-1} and \ref{chain-2} imply $\Gamma_{1,3}^3=\{\Gamma_{3,1},\Gamma_{\wz{a}},\Gamma_{\wz{b}}\}$ with $k_{3,1}=4$, $k_{\wz{b}}=12$ and $k_{\wz{a}}\in\{1,2,4\}$. Hence, $(z_0,z_4)\in\Gamma_{2,2}$.

Since $p_{(1,3),(1,3)}^{(2,2)}=2$, we may assume that $P_{(1,3),(1,3)}(z_0,z_4)=\{z,z'\}$ and $w\in P_{(1,3),(1,3)}(z_4,z_0)$. By Steps \ref{chain-1} and \ref{chain-2}, one gets $z_1,z_1'\in\Gamma_{\wz{b}^{*}}(z_4)$ with $k_{\wz{b}}=12$. The fact $k=4$ implies that $z,z',z_1,z_1'$ are pairwise distinct. In view of Step \ref{step5}, we have $z,z'\in P_{(2,2),(3,1)}(w,z_0)$. Since $p_{(2,2),(3,1)}^{(1,3)}=3$ from Step \ref{step2}, we may assume $(w,z_1)\in\Gamma_{2,j}$ and $(w,z_1')\in\Gamma_{2,2}$. In view of Step \ref{chain-1}, one has $(w,z_2)\in\Gamma_{\wz{a}}$ with $k_{\wz{a}}\in\{1,2,4\}$. If $(z_3,w)\in\Gamma_{2,2}$, from Steps \ref{chain-1} and \ref{chain-2},  then $(z_2,w)\in\Gamma_{\wz{b}}$ with $k_{\wz{b}}=12$, a contradiction. By Step \ref{step5}, we obtain $(z_3,w)\in\Gamma_{2,j}$, which implies $(z_2,w)\in\Gamma_{\wz{a}}$ from Step \ref{chain-1}. It follows that $\wz{a}=\wz{a}^*$. Since $k_{2,2}=6$ from Step \ref{step4}, we have $\wz{a}=(3,3)$.

Since $p_{(2,j),(3,1)}^{(1,3)}=1$, we have $z_1',z_3'\in\Gamma_{2,2}(w)$. Since $(z_0,z_2''),(z_2'',z_4)\in\Gamma_{2,j}$, from Steps \ref{chain-1} and \ref{chain-2}, one gets $(w,z_2''),(z_2'',w)\in\Gamma_{\wz{b}}$ with $k_{\wz{b}}=12$. It follows that $\wz{b}=\wz{b}^*$. The fact $k_{2,2}=6$ implies $\wz{b}=(3,3)$, a contradiction.

{\bf Case 2.} $p_{(2,2),(2,2)}^{\wz{c}}\geq2$.

Note that there exists $z_2'''\in P_{(2,2),(2,2)}(z_0,z_4)$ with $z_2'''\notin\{z_2,z_2',z_2''\}$. Since $p_{(1,3),(1,3)}^{(2,2)}=2$, we may assume $P_{(1,3),(1,3)}(z_0,z_2''')=\{z_1'',z_1'''\}$ and $P_{(1,3),(1,3)}(z_2''',z_4)=\{z_3'',z_3'''\}$. Note that $z_2'\in P_{(1,3),(3,1)}(z_1,z_1')\cap P_{(3,1),(1,3)}(z_3,z_3')$. Since $p_{(1,3),(3,1)}^{\wz{l}}\leq1$ for $\wz{l}\neq(0,0)$, one has $\{z_1'',z_1'''\}\neq\{z_1,z_1'\}$ and $\{z_3'',z_3'''\}\neq\{z_3,z_3'\}$. Without loss of generality, we may assume $z_1''\notin\{z_1,z_1'\}$ and $z_3''\notin\{z_3,z_3'\}$.

If $z_1'''=z_1'$, from Steps \ref{chain-1} and \ref{chain-2}, then $(z_1',z_4)\in\Gamma_{\wz{b}}$, which implies $z_2',z_2'''\in P_{(1,3),(2,2)}(z_1',z_4)$, contrary to the fact that $p_{(2,2),(1,3)}^{\wz{b}}=1$. Similarly, we have $z_1'''\notin\{z_1,z_1'\}$ and $z_3'''\notin\{z_3,z_3'\}$.

Since $p_{(1,3),(3,1)}^{\wz{l}}\leq1$ for $\wz{l}\neq(0,0)$, we have $z_3''$ or $z_3'''\notin\Gamma_{3,1}(z_0)$. Without loss of generality, we may assume $(z_0,z_3''')\notin\Gamma_{3,1}$. By Step \ref{jb g=4}, we may assume $(z_1'',z_3''')\in\Gamma_{2,j}$. Since $p_{(2,j),(3,1)}^{(1,3)}=1$, we obtain $(z_1''',z_3''')\in\Gamma_{2,2}$. The fact $p_{(1,3),(1,3)}^{(2,2)}=2$ implies that there exists $z_2''''\in P_{(1,3),(1,3)}(z_1''',z_3''')$ with $z_2''''\neq z_2'''$. Since $p_{(1,3),(1,3)}^{(2,j)}=1$, one has $z_2''''\notin\{z_2,z_2',z_2'',z_2'''\}$. Note that $(z_1'',z_3''')\in\Gamma_{2,j}$ and $(z_0,z_3''')\notin\Gamma_{3,1}$. By Step \ref{jb g=4}, we get $(z_0,z_2'''')\in\Gamma_{2,j}$.

Suppose $(z_2'''',z_4)\in\Gamma_{2,2}$. The fact $z_2'''\in P_{(3,1),(1,3)}(z_3''',z_3'')$ and $p_{(1,3),(3,1)}^{\wz{l}}\leq1$ for $\wz{l}\neq(0,0)$ imply that $(z_2'''',z_3'')\notin\Gamma_{1,3}$. Since $p_{(1,3),(1,3)}^{(2,2)}=2$ and $\Gamma_{3,1}(z_4)=\{z_3,z_3',z_3'',z_3'''\}$, one gets $z_3$ or $z_3'\in\Gamma_{1,3}(z_2'''')$. By Steps \ref{chain-1} and \ref{chain-2}, we have $z_3,z_3'\in\Gamma_{\wz{b}}(z_0)$. It follows that $z_2'''',z_2\in P_{(2,j),(1,3)}(z_0,z_3)$ or $z_2'''',z_2''\in P_{(2,j),(1,3)}(z_0,z_3')$, contrary to the fact that $p_{(2,j),(1,3)}^{\wz{b}}=1$ from Step \ref{chain-2}. Then $(z_2'''',z_4)\in\Gamma_{2,j}$.

By Steps \ref{chain-1} and \ref{chain-2}, we have $z_3,z_3',z_3'',z_3'''\in P_{\wz{b},(1,3)}(z_0,z_4)$, which implies $p_{\wz{b},(1,3)}^{\wz{c}}=4$. In view of Lemma \ref{jb} (ii) and Step \ref{chain-2}, one obtains $4k_{\wz{c}}=p_{\wz{b},(1,3)}^{\wz{c}}k_{\wz{c}}=k_{\wz{b}}p_{\wz{c},(3,1)}^{\wz{b}}=12p_{\wz{c},(3,1)}^{\wz{b}}$, which implies $k_{\wz{c}}=3p_{\wz{c},(3,1)}^{\wz{b}}$. Since $z_2,z_2'',z_2''''\in P_{(2,j),(2,j)}(z_0,z_4)$, one gets $p_{(2,j),(2,j)}^{\wz{c}}\geq3$. Since $p_{\wz{c},(j,2)}^{(2,j)}\leq4$ from Lemma \ref{jb} (iii), by \eqref{eq-1}, one has $k_{\wz{c}}=3$.

Since $\Gamma_{1,3}^2=\{\Gamma_{2,2},\Gamma_{2,j}\}$ with $k_{2,j}=4$, from Steps \ref{chain-1} and \ref{chain-2}, we get $\Gamma_{1,3}^3=\{\Gamma_{3,1},\Gamma_{\wz{a}},\Gamma_{\wz{b}}\}$ with $k_{3,1}=4$, $k_{\wz{b}}=12$ and $k_{\wz{a}}\in\{1,2,4\}$. Then $\partial(z_0,z_4)=4$.

This completes the proof of this step.

\vspace{2.9ex}

In the following, we reach a contradiction based on the above discussion.

By Steps \ref{step4}, \ref{step5} and \ref{step7}, one gets $\Gamma_{1,3}^2=\{\Gamma_{2,2},\Gamma_{2,j}\}$ with $k_{2,2}=6$. In view of Steps \ref{chain-1} and \ref{chain-2}, we have $\Gamma_{1,3}^3=\{\Gamma_{3,1},\Gamma_{\wz{a}},\Gamma_{\wz{b}}\}$ with $k_{\wz{b}}=12$ and $k_{\wz{a}}\in\{1,2,4\}$. Step \ref{chain-3} implies $k_{4,4}\in\{3,6\}$. Since $p_{(1,3),(3,1)}^{\wz{l}}\leq1$ for $\wz{l}\neq(0,0)$ from Step \ref{step6}, we have $\Gamma_{2,2},\Gamma_{4,4}\notin\Gamma_{1,3}\Gamma_{3,1}$ by Lemma \ref{jb} (v). Lemma \ref{jb2} implies $\Gamma_{1,3}\Gamma_{3,1}\subseteq\{\Gamma_{0,0},\Gamma_{2,j},\Gamma_{j,2},\Gamma_{\wz{a}},\Gamma_{\wz{a}^*},\Gamma_{\wz{b}},\Gamma_{\wz{b}^*}\}$.

Suppose $\Gamma_{\wz{b}}\in\Gamma_{1,3}\Gamma_{3,1}$. Since $k_{\wz{b}}=12$, from Lemma \ref{jb} (i), one gets $\Gamma_{1,3}\Gamma_{3,1}=\{\Gamma_{0,0},\Gamma_{\wz{b}}\}$. It follows from Step \ref{chain-2} that there exists a path $(x_0,x_1,x_2,x_3,x_4)$ such that $(x_0,x_2)\in\Gamma_{2,j}$, $(x_1,x_3)\in\Gamma_{2,2}$ and $(x_0,x_4)\in\Gamma_{1,3}$. By Step \ref{chain-3}, we obtain $(x_2,x_4)\in\Gamma_{2,2}$. Since $p_{(1,3),(1,3)}^{(2,2)}=2$, there exists $x_3'\in P_{(1,3),(1,3)}(x_2,x_4)$ with $x_3\neq x_{3}'$. The fact $\Gamma_{1,3}\notin\Gamma_{1,3}\Gamma_{3,1}$ implies $(x_1,x_4)\notin\Gamma_{3,1}$. By Step \ref{jb g=4}, one has $(x_1,x_3')\in\Gamma_{2,j}$. In view of Steps \ref{chain-1} and \ref{chain-2}, we get $(x_0,x_3')\in\Gamma_{\wz{a}}$ with $\wz{a}\neq\wz{b}$. Since $x_4\in P_{(1,3),(3,1)}(x_0,x_3')$, one obtains $\Gamma_{\wz{a}}\in\Gamma_{1,3}\Gamma_{3,1}$, contrary to the fact that $\Gamma_{1,3}\Gamma_{3,1}=\{\Gamma_{0,0},\Gamma_{\wz{b}}\}$.

Suppose $\Gamma_{\wz{b}}\notin\Gamma_{1,3}\Gamma_{3,1}$. Note that $\Gamma_{1,3}\Gamma_{3,1}\subseteq\{\Gamma_{0,0},\Gamma_{2,j},\Gamma_{j,2},\Gamma_{\wz{a}},\Gamma_{\wz{a}^*}\}$. Step \ref{chain-1} implies $k_{\wz{a}}\in\{1,2,4\}$. Since $k_{2,j}=4$ and $j>2$, from Step \ref{step6} and Lemma \ref{jb} (i), (vi), we have $\Gamma_{1,3}\Gamma_{3,1}=\{\Gamma_{0,0},\Gamma_{2,j},\Gamma_{j,2},\Gamma_{\wz{a}}\}$ with $k_{\wz{a}}=4$. The fact $k_{2,2}=6$ implies $\wz{a}=(3,3)$. Let $(x_0,x_1,x_2,x_3)$ be a path such that $(x_0,x_2)\in\Gamma_{2,j}$ and $(x_0,x_3)\in\Gamma_{1,3}$. Since $x_0\in P_{(3,1),(1,3)}(x_1,x_3)$ and $\Gamma_{1,3}\Gamma_{3,1}=\{\Gamma_{0,0},\Gamma_{2,j},\Gamma_{j,2},\Gamma_{3,3}\}$, we have $(x_1,x_3)\in\Gamma_{2,j}$. Step \ref{chain-1} implies $(1,3)=\wz{\partial}(x_0,x_3)=\wz{a}=(3,3)$, a contradiction.

\section*{Acknowledgements}
Y.~Yang and Z.~Wang are supported by the Fundamental Research Funds for the Central Universities (Grant No.~2652019319).

\end{CJK*}

\end{document}